\documentclass[10pt]{article}
\usepackage[french]{babel}
\usepackage{latexsym}
\usepackage{amsmath,amsxtra,amssymb,latexsym, amscd,amsthm,}
 \usepackage[pdftex]{color,graphicx}
\usepackage{pgf,tikz}
\usepackage[T1]{fontenc}
\usepackage{cleveref}
\usepackage{mathrsfs}
\usepackage{amsfonts}
\usetikzlibrary{arrows}
\usepackage{graphicx}
\usepackage[]{algorithm2e}	
\usepackage{caption}
\usepackage{subcaption}
\usepackage{graphpap}
\usepackage{rotating}
\theoremstyle{plain}
\newtheorem{corollary}{Corollary}[section]
\newtheorem{thm}{Theorem}
\newtheorem{lemma}{Lemma}[section]
\newtheorem{question}{Question}

\theoremstyle{definition}

\newtheorem{prop}{Proposition}
\newtheorem{definition}{Definition}[section]
\newtheorem{remark}{Remark}[section]
\newtheorem{remarks}{Remarks}[section]

\newcommand{\R}{\mathbb{R}}
\newcommand{\Z}{\mathbb{Z}}
\newcommand{\N}{\mathbb{N}}
\newcommand{\sg}{\Sigma}
\newcommand{\gm}{\Gamma}

\title{\sc{Intersection Norms and One-faced Collections of Curves}}
\author{Abdoul Karim SANE}
\date{ ENS Lyon, October 15th 2018.}

\begin{document}
\renewcommand{\proofname}{Proof}
\renewcommand{\refname}{Bibliography}
\renewcommand{\abstractname}{Abstract}
\maketitle

\begin{abstract} Intersection norms are integer norms on the first homology group of a surface. In this article, we prove that there are some polytopes which are not dual unit balls of such norms. By the way, we investigate the set of collections of curves on $\sg_2$ whose complement is a disk.\footnote{\scriptsize{Research supported by the laboratory UMPA-ENS LYON}} 
\end{abstract}
\begin{section}{Introduction}
Intersection norms on surfaces were first quickly introduced by Turaev ~\cite{turaev} (page 143), and  studied by M. Cossarini and P. Dehornoy ~\cite{Direc}. They use intersection norms to classify up to isotopy all surfaces transverse to the geodesic flow on the complement of special links in the unit tangent bundle of a closed oriented surface. \\
Their result makes explicit Thurston's fibered faced theory for Thurston norms on compact oriented 3-manifolds. It tells us that an intersection norm on a surface (respectively the Thurston norm on a 3-manifold) encodes the open book decompositions of the unit tangent bundle of that surface (respectively the topology of a fibered ~3-manifold ).

Our purpose in this article is to study intersection norms for their own.\\
Let $\sg_g$ be a closed oriented surface of genus $g\geq 1$, and $\gm$ a collection of closed curves on ~$\sg_g$. We assume that $\gm$ has only double intersection points. Let $\alpha$ be loop on $\Sigma_g$, we define the number $i_{\gm}(\alpha)$ as follows:
$$i_{\gm}(\alpha)=\inf\{\#\{\alpha'\cap\gm\}; \alpha'\sim\alpha; \alpha'\pitchfork\gm\};$$
where the symbol $\sim$ (respectively $\pitchfork$) is the free homotopy relation (respectively transversality).\vspace{0,5cm}

We define
\begin{align*}
N_{\Gamma}:H_1(\sg_g,\mathbb{R})&\longrightarrow \mathbb{R}\\
                     a&\longmapsto \inf\{i_{\gm}(\alpha);[\alpha]=a\}.
\end{align*} 
The function $N_{\gm}$ defines a semi-norm on $H_1(\Sigma_g,\mathbb{R})$ and it takes integer values on the lattice $H_1(\Sigma_g,\mathbb{Z})$.  Using a standard basis for the homology, we shall identify ~$H_1(\sg_g,\R)$ and ~$H^1(\sg_g,\R)$ with $\R^{2g}$.  By a theorem of Thurston \cite{thurst}, the dual unit  ball of $N_{\gm}$ is a lattice polytope, ie, the convex hull of finitely many integer vectors (by integer vector, we mean a vector in the integer lattice ~$H^1(\sg_g,\Z)$).\\
Moreover, if $\gm$ fills $\sg_g$ that is $\Sigma_g-\gm$ is a union of topological disks, then ~$N_{\Gamma}$ defines a norm, i.e, its dual unit  ball has non empty interior in ~$H^1(\sg_g,\R)$. 
 
 One constraint on the dual unit  balls of intersection norms is that their vertices are congruent modulo 2. This comes from the fact that geometric and  algebraic intersections have the same parity. In genus $1$, this constraint happens to be the only one. So, every symmetric convex lattice polygon with $\rm{mod} \hspace{0,1cm}2$ congruent vertices is the dual unit ball of an intersection norm on the torus. The proof of this fact  follows from an implicit argument in Thurston's paper ~\cite{thurst}. We will explain it in Section~\ref{sect1} for completeness.
 
 Now we raise the following problem\string: 
 
 \begin{question} Fix $g\geq 2$, and let $P\subset H^1(\sg_g,\R)$ be a symmetric lattice polytope all of whose vertices are congruent mod 2. Is it the dual unit ball of some intersection norm on $\sg_g$?
 \end{question} 
 This question is natural when we deal with integer norms coming from topology (for instance, we have an analogue of this question for the Thurston's norm).
 
 In this article, we give examples of lattice polytopes on $\R^4$ with $\rm{mod}\hspace{0,1cm}2$ congruent vertices,  which are not dual unit  balls of intersection norms. More precisely, we show that sub-polytopes --with eight vertices and non-empty interior-- of the cube $[-1,1]^4$, are not the dual unit balls of intersection norms. It means that in higher dimension, dual unit balls of intersection norms come with other constraints.

 Let $\mathcal{P}_8$ be the set of all symmetric sub-polytopes of $[-1,1]^4$ having eight vertices and non-empty interior. The set $\mathcal{P}_8$ is not empty; it contains the polytope generated by the following vectors (and their opposites)\string:
$$v_1=(1,1,1,1),\hspace{0.2cm }v_2=(1,-1,1,1),\hspace{0.2cm } v_3=(-1,1,1,1),\hspace{0.2cm } v_4=(1,1,-1,1).$$

Now, we state the main result of this article\string:
\begin{thm}\label{thmprinc}
Elements of $\mathcal{P}_8 $  are not dual unit  balls of intersection norms.
\end{thm}

If $\gm$ is a filling collection of curves on a surface, whose complement is a disk, we say that $\gm$ is a \textbf{\textit{one-faced collection}}. 

They have been called \textbf{minimally intersecting filling collections} in \cite{Aoug} and also \textbf{unicellular maps} in \cite{Chap}. \vspace{0,5cm} 

The proof of Theorem \ref{thmprinc} relies on\string: 

\begin{thm}\label{counting}
On a closed genus $2$ surface, there are  four orbits of one-faced collections whose dual unit balls are in the cube $[-1,1]^4$, under the mapping class group action.  
\end{thm}

The proof of Theorem $1$  uses the natural (partial) order on the set of functions. We relate that partial order to a topological operation on collections of closed curves and we use it to show that if an element of $\mathcal{P}_8$ is the dual unit ball of an intersection norm, then it must come from a one-faced collection ~$\gm$. Finally, we check that none of the four collections of Theorem~\ref{counting} realizes an element of $\mathcal{P}_8$.
\begin{paragraph}{Organization of this article\string:} In Section ~\ref{sect1}, we recall some facts on intersection norms and we show that for the question of realizability, we can restrict our attention to minimal collections. 

In Section ~\ref{sect2}, we show that any intersection norm is bounded from below by a norm defined by a one-faced collection. 

Finally, in Section ~\ref{sect3}, we count orbits (under the mapping class group action) of one-faced collections (whose dual unit balls are sub-polytopes of the cube ~$[-1,1]^4$) on $\sg_2$ and we prove Theorem \ref{thmprinc}.   
\end{paragraph}

\begin{paragraph}{Acknowledgments\string:}
I am very thankful to my two supervisors P. Dehornoy and J.-C Sikorav for careful reading and discussion at every step of the writing of this article. 
\end{paragraph}
\end{section}
\begin{section}{Preliminaries on intersection norms}\label{sect1}

In this section, we first recall some facts about integer (semi)-norm. Then, we define the intersection (semi)-norm define by a collection of curves and we recall some basic notions about them (For more details on intersection norms, see \cite{Direc}). We end this section by proving that, concerning the realizability of polytopes, we can restrict our attention to so-called minimal collections. 

Let $\sim$ denote the free homotopy relation on curves, $\pitchfork$ be transversality relation and $[.]$ the homology class. \vspace{0.2cm}

\begin{paragraph}{Integer norms\string:}
Let $E$ be a vector space of dimension $n$ and $$L=L(u_1,...,u_n):=\{a_1u_1+...+a_nu_n, a_i\in\Z\}$$  the lattice generate by the vectors $(u_i)_{i=1,...,n}$. 
\begin{definition}[\textit{Integer norm}]\leavevmode

A norm $N:E\longrightarrow \R_+$ is an \textit{\textbf{integer semi-norm}} relatively to the lattice ~$L$ if the restriction of $N$ to $L$ takes positive integer values. 
\end{definition}

The following theorem states that the dual unit ball of an integer-norm have a combinatorial description. 
\begin{thm}[W.Thurston]\label{thmthurst}
If $N$ is an integer semi-norm relatively to a lattice ~$L$, then its dual unit  ball is a convex hull of finitely many vectors in the lattice; $$B_{N^*}=ConvHull\{v_1,......, v_n; v_i\in L\}.$$
\end{thm}\vspace{0.5cm}
One can find a sketch of proof of Theorem~\ref{thmthurst} in \cite{thurst}  (Page 107-112). For a more complete proof, see \cite{flp} (ExposÈ Fourteen by David Fried). More recently, de la Salle gives a new proof of Theorem~\ref{thmthurst} (see \cite{lassal}).  
\end{paragraph}

\begin{paragraph}{Definition of intersection norms:}
We consider a genus $g$ closed oriented surface $\sg_g$ and a collection $\gm=\{\gamma_1,...,\gamma_n\}$   of closed curves on $\sg_g$. We insist on the fact that $\gm$ is fixed up to isotopy. Let ~$a\in H_1(\sg_g,\Z)$ be a homology class and $\alpha$  an oriented multi-curves representing $a$. Then we define:

$$i_{\gm}(\alpha):=\inf\{\#\{\alpha'\cap\gm\}; \alpha'\sim\alpha; \alpha'\pitchfork\gm\}$$

and

$$ N_{\gm}(a):=\inf\{i_{\gm}(\alpha); [\alpha]=a\}.$$

If a multi-curves $\alpha$ representing a homology class $a$ is such that $$N_{\gm}(a)=i_{\gm}(\alpha),$$ then $\alpha$ is $\gm$-\textbf{\textit{minimizing}}.

One important thing is that $\gm$-minimizing multi-curves can be chosen to be simple. In fact, if $\alpha$ is a (a priori non simple) $\gm$-minimizing multi-curve then by smoothing all the self-intersection points of $\alpha$ with respect to its orientation, we get a new oriented multi-curve $\alpha'$ in the same homology class as $\alpha$ and $i_{\gm}(\alpha')=i_{\gm}(\alpha)$. It implies that $\alpha'$ is a simple $\gm$-minimizing multi-curve as we claim.  

\begin{prop}\label{deh}
The function $N_{\gm}:H_1(\sg_g,\Z)\longrightarrow \N$ satisfies:\vspace{0,2cm}
\begin{itemize}
\item \textbf{linearity on rays}: $N_{\gm}(na)=|n|N_{\gm}(a)$ for all $n\in\Z$ and $a\in H_1(\sg_g,\Z)$
\item \textbf{convexity}: $N_{\gm}(a+b)\leq N_{\gm}(a)+N_{\gm}(b)$ for all $a, b\in H_1(\sg_g,\Z).$
\end{itemize}
\end{prop}\vspace{0,2cm}
The proof of Proposition~\ref{deh} is not trivial and one can see ~\cite{Direc}.\vspace{0,2cm}\vspace{0,2cm}

Linearity on rays implies that $N_{\gm}$ can be extended to homology with rational coefficients since for all $a\in H_1(\sg_g,\Z)$ and $q\in\N$, we have:
$$N_{\gm}(a)=N_{\gm}(\frac{q}{q}.a)=qN_{\gm}(\frac{1}{q}a).$$

It follows by convexity that $N_{\gm}$ extends uniquely to a positive function on ~$H_1(\sg_g,\R)$. Moreover, the extended function $N_{\gm}: H_1(\sg_g,\R)\longrightarrow \R_+$ is still  linear on rays and convex. Therefore, $N_{\gm}$ defines a semi-norm on $H_1(\sg_g,\R)$ and it takes integer values on the lattice $H_1(\sg_g,\Z)$. So, $N_{\gm}$ is an \textit{\textbf{integer semi-norm}}. Theorem ~\ref{thmthurst} implies that  the  dual unit ball $B_{N^*_{\gm}}$ is a convex hull of finitely many integer vectors. \vspace{0,1cm} 

If the collection is filling, then $N_{\gm}$ defines an \textit{\textbf{integer norm}}.  
\end{paragraph}

\begin{paragraph}{Relation between the vectors of the dual unit ball\string:}
If $\alpha$ and $\beta$ are two transverse oriented closed curves , then the algebraic intersection number between $\alpha$ and $\beta$ is given by 

 $$i_a(\alpha,\beta)=\displaystyle{\sum_{p\in\alpha\cap\beta}{\varepsilon(p,\alpha,\beta)}};$$
where $\varepsilon(p, \alpha, \beta)$ is the algebraic sign of the intersection at $p$, relatively to the orientation of $\sg_g$. We recall that $i_a$ depend only on the homology classes of ~$\alpha$ and $\beta$, and defines a non degenerate skew-symmetric 2-form on $H_1(\sg_g,\R)$.\vspace{0,2cm}  

Then, if $\alpha$ and $\alpha'$ are two homologous curves, by taking an orientation of ~$\gm$, we have  

 $$i_{\gm}(\alpha)=i_a(\alpha,\gm)\hspace{0,1cm}\rm{mod}\hspace{0,1cm}2;$$  $$i_{\gm}(\alpha')=i_a(\alpha',\gm)\hspace{0,1cm}\rm{mod}\hspace{0,1cm}2;$$  $$i_a(\alpha,\gm)=i_a(\alpha',\gm).$$\vspace{0,2cm}
 
 Thus, $i_{\gm}(\alpha)=i_{\gm}(\alpha')\mod2$ for every orientation of $\gm$. Therefore, if $v_1$ and $v_2$ are two integer vertices in the dual unit sphere of $N_{\gm}$, $$v_1=v_2 \mod2.$$ 
 
The relation above is a necessary condition for a symmetric lattice polytope to be the dual unit ball of an intersection norm.  The following statement shows that it is sufficient in the genus one case. The idea of the proof is from Thurston ~\cite{thurst}.
 
 \begin{prop} If $P$ is a symmetric lattice polygon in the plane with congruent mod $2$ vertices, then $P$ is the dual unit ball of an intersection norm. 
\end{prop}
\begin{proof}
First, if $P$ is a symmetric lattice segment in $\R^2$, then there is a matrix ~$A\in \rm{SL}(2,\Z)$ such that $P':=A(P)$ is a vertical segment with extremities in $\Z^2$. Moreover, $A$ has a geometric realization since $\rm{Mod}(\mathbb{T}^2)=SL(2,\Z)$. That is there is a homeomorphism $\phi$ of $\mathbb{T}^2$ such that $$\phi_*:H_1(\mathbb{T}^2,\R)\approx \R^2\longrightarrow H_1(\mathbb{T}^2,\R)\approx \R^2$$ is equal to $A$.\\ Now, let $l:=\frac{1}{2}\rm{length}(P')$; $l\in\N$. If $\alpha$ and $\beta$ are the canonical basis of ~$H_1(\mathbb{T}^2,\R)$, by taking $l$ parallel curves to $\beta$, we get a collection $\gm'$ in $\mathbb{T}^2$ such that $B_{N^*_{\gm'}}=P'$. So, $\gm:=\phi^{-1}(\gm')$ is such that $B_{N^*_{\gm}}=P$.\vspace{0,5cm}

Secondly, if $\gm:=\{\gamma_1,...,\gamma_n\}$ is a collection of closed geodesics on $\mathbb{T}^2$ (with the flat metric of constant curvature equal to 1), and if $a$ is a homology class represented by a collection $\alpha$ of oriented simple closed curves which are pairwise disjoint then
$$N_{\gm}(a)=i_{\gm}(\alpha)=\displaystyle{\sum_{j=1}^n{i_{\gamma_j}(\alpha)}}=\displaystyle{\sum_{j=1}^n{N_{\gamma_j}(a)}}.$$ 

It follows that the dual unit ball of $N_{\gm}$ is equal to the Minkowski sum of the dual unit balls of $N_{\gamma_j}$; which are symmetric lattice segments\string:
$$B_{N^*_{\gm}}=\bigoplus_j{B_{N^*_{\gamma_j}}}.$$

Finally, every symmetric lattice polygon of $\R^2$ is the Minkowski sum of finitely many symmetric lattice segments.

Combining the three arguments above we construct, for any symmetric lattice polygon $P$ with mod $2$ congruent vertices, a geodesic collection $\gm$ such that
$$B_{N^*_{\gm}}= P.$$

\end{proof}

\end{paragraph}

\begin{paragraph}{Minimality of the collection\string:} Now, we  show that we can restrict to collections in minimal position. 

\begin{definition}
Let $\gamma_1$ and $\gamma_2$ be two transverse closed curves on $\sg_g$. They are in \textit{\textbf{minimal position}} if they realize the geometric intersection in their free homotopy classes that is $$i(\gamma_1,\gamma_2)=\rm{card}\{\gamma_1\cap\gamma_2\}.$$

A collection $\gm$ is \textit{\textbf{minimal}} if all the curves in $\gm$ are pairwise in minimal position. 
\end{definition}

\begin{remark} One-faced collections are minimal.
\end{remark}

\begin{lemma}\label{mincollect}
Let $\gm$ be a collection of closed curves in $\sg_g$, then there is a minimal collection $\gm_{\rm{min}}$ such that $N_{\gm}=N_{\gm_{\rm{min}}}$.
\end{lemma}
\begin{proof} One can apply a generic homotopy so that we get a collection in minimal position. Such a generic homotopy consists to do a finite number of Reidemester moves (1, 2 and 3 as depicted in Table 1). By Hass and Scott ~\cite{HasScott}, one can choose a decreasing homotopy with respect to the intersection number of the collection. Moves 1 and 3 do not change the norm, while Move 2 (deleting a bigon) changes the norm.

Then, we replace Move $2$ by a crossing (see Table~\ref{reid}). This new move changes the homotopy class of $\gm$ but it does not change the norm. By changing deleting bigon by a crossing, the self-intersection decreases by one. As we can choose a descending homotopy, we get a collection $\gm_{\rm{min}}$ in minimal position after applying finitely many Reidemester's move $1$ and $3$ and crossing move. Doing so, the norm does not change; hence we built a new collection $\gm_{\rm{min}}$ in minimal position such that  $N_{\gm}=N_{\gm_{\rm{min}}}$.   
\end{proof}

\begin{table}[h!]
\begin{center}
\begin{tabular}{cccc}
\includegraphics[scale=0.11]{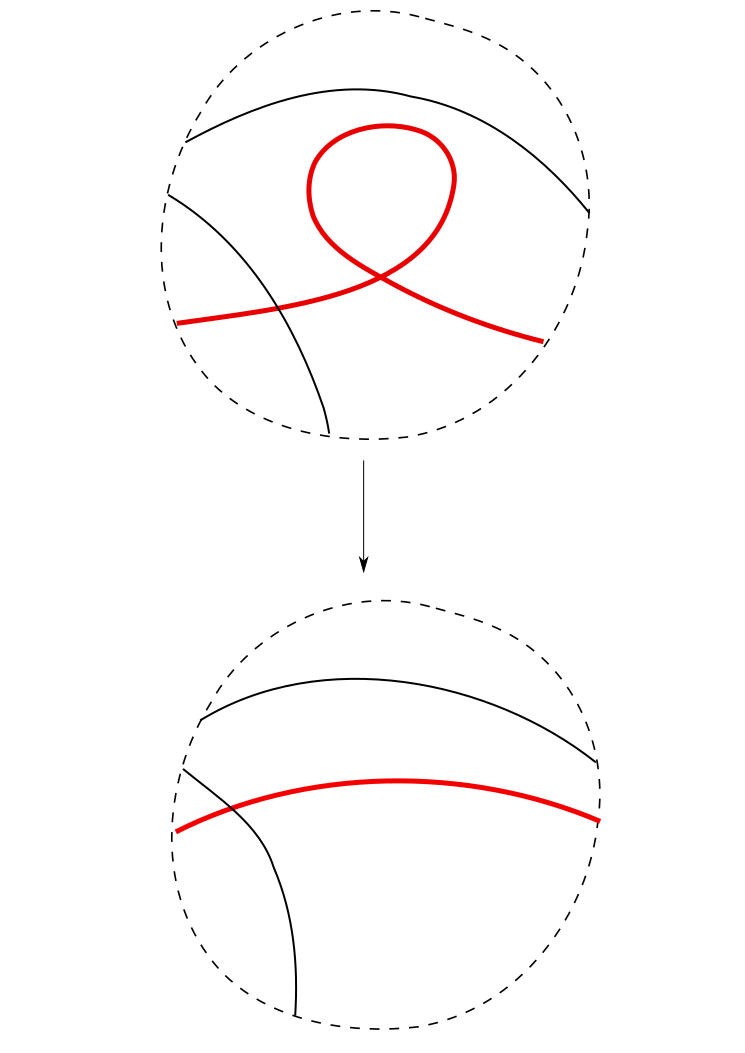}

&\includegraphics[scale=0.12]{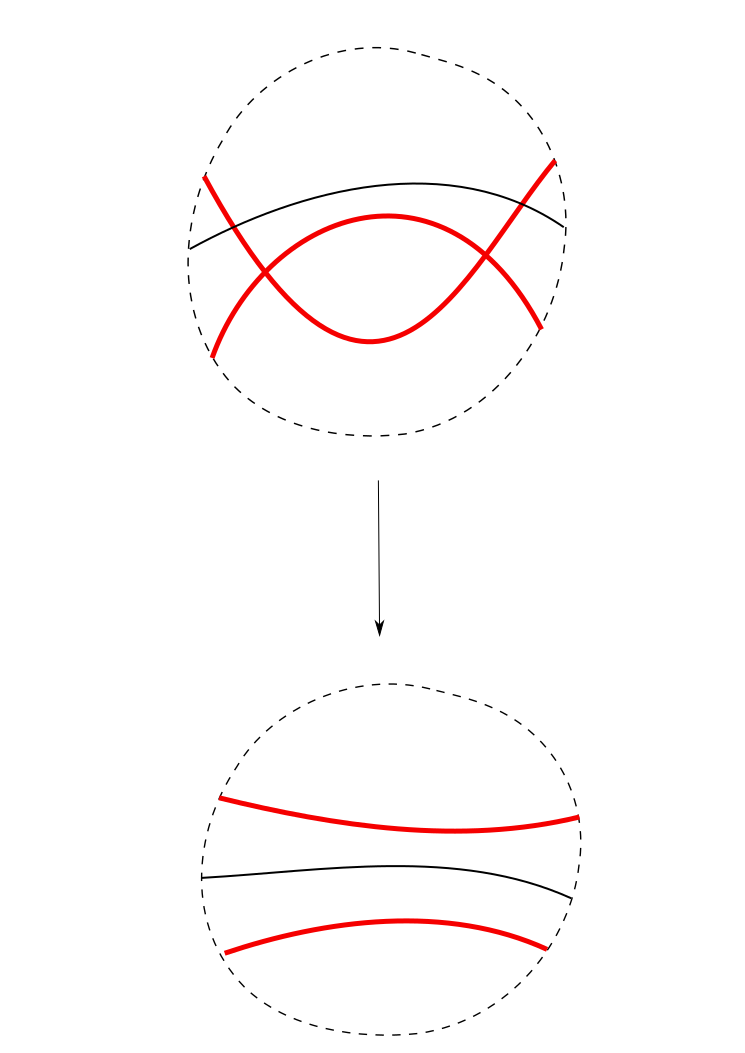}

\includegraphics[scale=0.11]{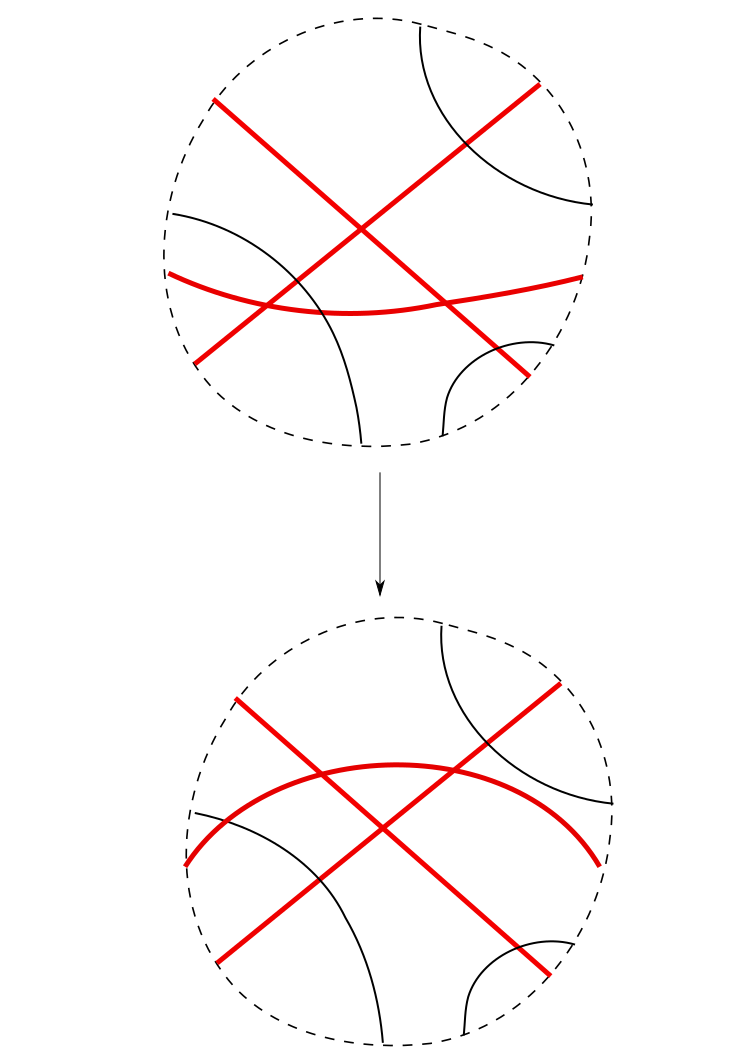}

&\includegraphics[scale=0.11]{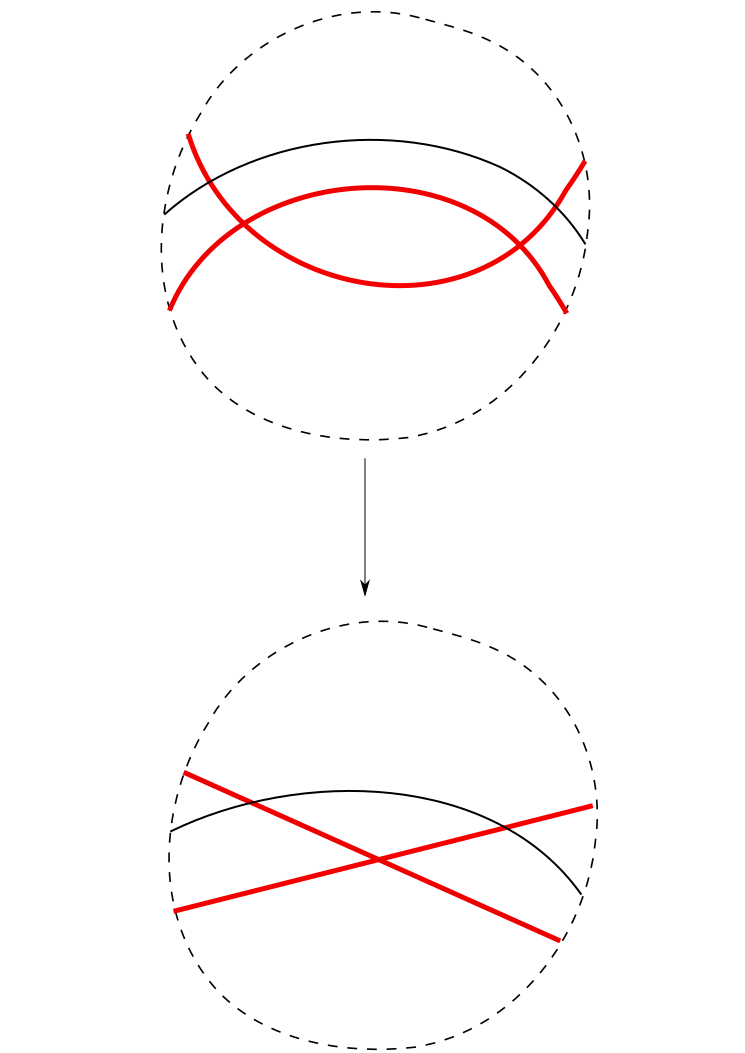}\\
\end{tabular}
\end{center}
\caption{From left to right, we have the three Reidemeister's moves and the last one is the crossing move. The curves in red color represents the local configuration of $\gm$, and curves in black are sub-arcs of curves in $\sg_g$. One can see that Reidemeister's move 1 and 3 and crossing move do not change the intersection. Otherwise, the bigon deleting does change the intersection.}
\label{reid}
\end{table} 
\end{paragraph}
\end{section}

\begin{section}{ Partial order and lower bound in the set of intersection norms on ~$\sg_g$}\label{sect2}

In this section, we first recall an important tool for computing the dual unit ball of intersection: \textit{\textbf{Eulerian co-orientation}}.\vspace{0,1cm}

Then we define a topological operation on collections of closed curves and relate it to the partial order on the set of all intersection norms.

We finish by proving that every intersection norm is bounded from below by an intersection norm induced by a one-faced collection.


\begin{subsection}{Eulerian Co-orientation\string:}

We consider a collection of closed curves $\gm$ on $\sg_g$ such that $\sg_g-\gm$ is a union of topological disks. The collection $\gm$ defines a filling graph on ~$\sg_g$. We denote by $V(\gm)$ the set of its vertices, defined as self-intersection points of $\gm$. Let $E(\gm)$ be the set of edges and ~$F(\gm)$ the set of faces. 

The Euler characteristic of $\sg_g$ is given by:
$$\chi(\sg_g)= 2-2g=	|V|-|E|+|F|.$$

\begin{definition}

A \textit{\textbf{co-orientation}} of $\Gamma$ is a choice of a positive way to cross (transversally) every edge of $\gm$. 

A co-orientation is \textit{\textbf{Eulerian}} if a small oriented circle centered at a vertex crosses positively two edges and negatively the other two, relatively to the co-orientation.
\end{definition}
\begin{figure}[htbp]
\begin{center}
\includegraphics[scale=0.15]{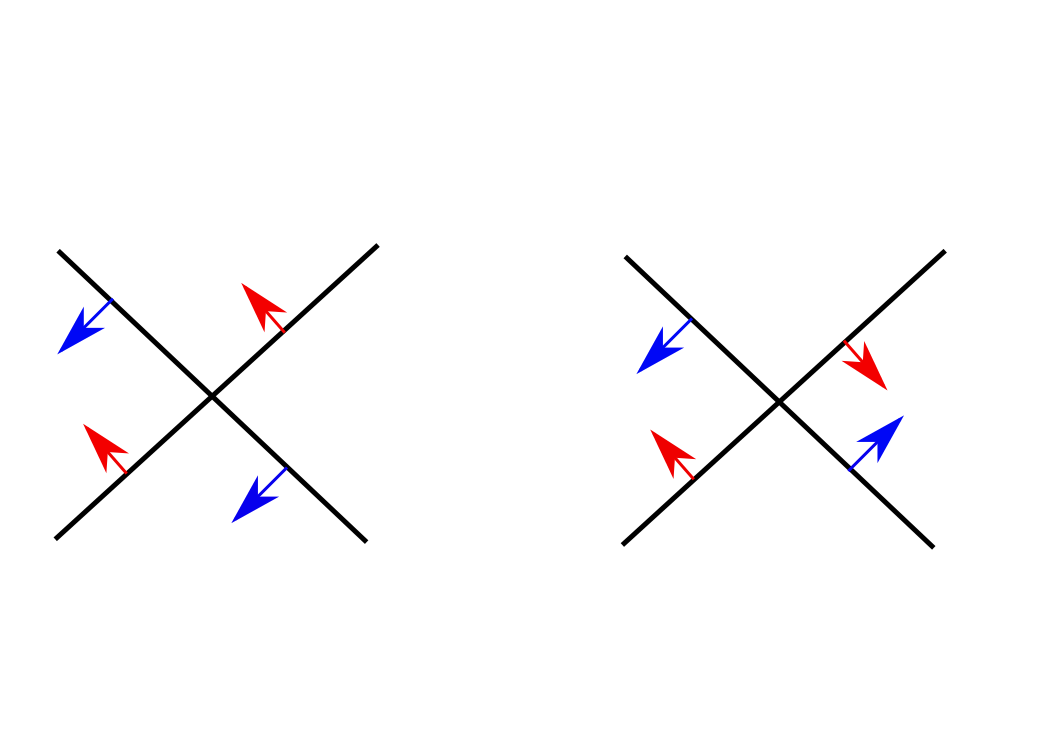}
\caption{Non alternating and alternating co-orientation.}
\label{corient}
\end{center}
\end{figure}

\begin{remarks}\leavevmode
\begin{itemize}
\item Up to rotation, we distinguish two types of Eulerian co-orienta\-tions around a vertex (see Figure~\ref{corient}). A vertex is \textit{non-alternating} if the arcs emanating from it, and belonging to the same curve are co-oriented in the same direction, otherwise it is \textit{alternating}.
\item To a co-orientation of an arc $A$ correspond an orientation of it. It is the one that gives, together with the co-orientation of $A$, the orientation of ~$\sg_g$.
\item A collection $\gm$ with $c$ curves has at least $2^c$ Eulerian co-orientations given by all the different ways to co-orient $\gm$ only by non-alternating vertices  (this is equal to the number of possibilities to orient $\gm$). 
\end{itemize}
\end{remarks}

Let $\alpha$ be an oriented closed curve on $\Sigma_g$ transverse to $\gm$, and let $\nu$ be a co-orientation of $\Gamma$. We define
\begin{equation*}
\nu(\alpha):=\displaystyle{\sum_{p\in \alpha\pitchfork\Gamma}\varepsilon(p,\alpha,\Gamma{\nu})}, 
\end{equation*}
where $\varepsilon(p,\alpha,\Gamma^{\nu})=\pm1$ depending on whether $\alpha$ crosses $\Gamma$ at $p$ in the direction of the co-orientation $\nu$ or not. Moreover, if $\nu$ is a Eulerian co-orientation, $$\nu(\alpha)=0\hspace{0,2cm} \rm{if}\hspace{0,2cm} [\alpha]=0.$$ Therefore, a Eulerian co-orientation $\nu$ defines a map 
\begin{align*}
[\nu]:H_1(\sg_g,\R)&\longrightarrow \R_+\\
        H_1(\sg_g,\Z)&\longrightarrow \N.\\             
\end{align*} 
So, a Eulerian co-orientation defines an integer cohomology class. We denote by $\rm{Eulco}(\gm)$ the set of all Eulerian co-orientations of $\gm$ and by $[\rm{Eulco}(\gm)]$ the set of their cohomology classes (different co-orientations can give the same cohomology class). 

\begin{thm}[M. Cossarini \& P. Dehornoy]\label{Pierre}
\textit{The set $[\rm{Eulco}(\gm)]$ is a subset of the unit dual ball $B_{N^*_{\gm}}$. Moreover, every integer vector in $B_{N^*_{\gm}}$, \rm{mod} 2 congruent to the vertices of $B_{N^*_{\gm}}$ belongs to $[\rm{Eulco}(\gm)]$}. 
\end{thm}

The proof of Theorem \ref{Pierre} is well explained in \cite{Direc}.
\end{subsection}


\begin{subsection}{Topological operation on collections of curves}
Now, we explain how we can topologically compare collections of curves.\\
Let $p$ be a self-intersection point of $\gm$. We construct two collections $\gm_1$ and ~$\gm_2$ by smoothing the collection $\gm$ at $p$ in the two possible ways (see Figure 2).

 After this smoothing, we get two collections, named $\widetilde{\gm_1}$ and $\widetilde{\gm_2}$, a priori not in minimal position. By Corollary 1.1, there are two minimal collections $\gm_1$ and ~$\gm_2$  such that $N_{\widetilde{\gm_1}}=N_{\gm_1}$ and $N_{\widetilde{\gm_2}}=N_{\gm_2}$.
 
 \begin{figure}[htbp]
\begin{center}
\includegraphics[scale=0.12]{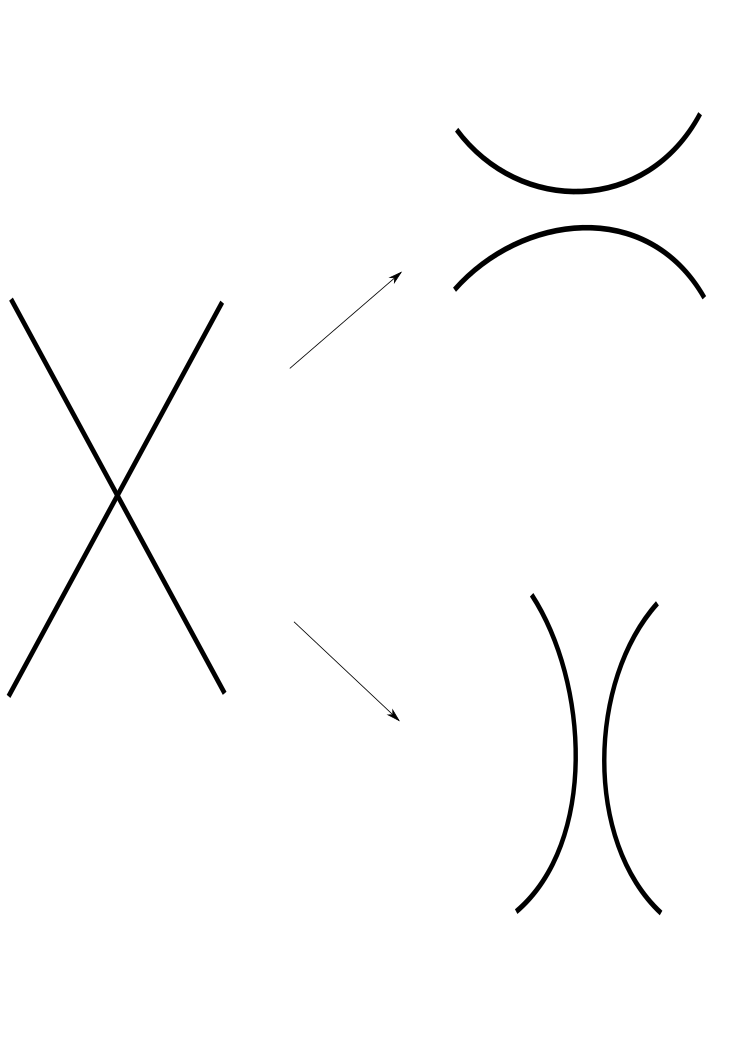}
\caption{Smoothing at a self-intersection point.}
\label{smth}
\end{center}
\end{figure}

 \begin{lemma}\label{smt}
 The collections $\widetilde{\gm_1}$ and $\widetilde{\gm_2}$ obtained by smoothing an intersection point of $\gm$ are such that 
 
 $$[\rm{Eulco}(\gm)]=[\rm{Eulco}(\widetilde{\gm_1})]\cup[\rm{Eulco}(\widetilde{\gm_2})].$$
 
In particular, we have $$B_{N^*_{\gm}}=conv(B_{N^*_{\gm_1}} \cup B_{N^*_{\gm_2}}).$$
\end{lemma}
\begin{proof}
Let $p$ be an intersection point  of $\gm$ and $\widetilde{\gm_1}$ one of the collections obtained by smoothing $\gm$ at $p$. The collection $\gm$ differs to $\gm_1$ only in a small neighborhood (which is a disk) of $p$. To a co-orientation $\nu_1$ of $\widetilde{\gm_1}$, we associate a co-orientation $\nu$ of $\gm$ in the following way: we keep the same co-orientation of $\widetilde{\gm_1}$ on $\gm$ outside a small neighborhood of $p$.

 Since $\nu_1$ is an Eulerian co-orientation, the boundary of the neighborhood around $p$ intersects two times $\nu_1$ in a positive direction and two times in a negative direction. It implies that the induced co-orientation on $\gm$ by ~$\nu_1$ outside the neighborhood of $p$ can be completed in a unique way in the neighborhood of $p$; the point $p$ at the end could be non-alternating or alternating. We then obtain an Eulerian co-orientation $\nu$ of $\gm$.

As the coordinates of a vector associated to a Eulerian co-orientation are obtained by evaluating this co-orientation on a basis of $H_1(\sg_g,\Z)$, one can choose the basis such that they do not enter in the neighborhood of $p$. It implies that $\nu$ and $\nu_1$ evaluated on that basis will give the same vector. Hence,  $ [\rm{Eulco}(\widetilde{\gm_1})]\subset[\rm{Eulco}(\gm)]$.

Now let $\nu$  be a Eulerian co-orientation of $\gm$. Then, by smoothing $\gm$ at ~$p$ with respect to  the co-orientation $\nu$ at $p$ (or with respect to the induced orientation of the four arcs emanating from $p$), we obtain a co-orientation $\nu_1$ of $\widetilde{\gm_1}$ or $\widetilde{\gm_2}$ which is equal to $\nu$ outside a small neighborhood of $p$. Therefore, $$[\nu]=[\nu_1].$$

So, we have   $$ [\rm{Eulco}(\gm)]\subset[\rm{Eulco}(\widetilde{\gm_1})]\cup[\rm{Eulco}(\widetilde{\gm_2})]$$ and the equality
 $$[\rm{Eulco}(\gm)]=[\rm{Eulco}(\widetilde{\gm_1})]\cup[\rm{Eulco}(\widetilde{\gm_2})]$$ 
 holds.
 
Finally, by Lemma ~\ref{mincollect} and Lemma ~\ref{Pierre}, we have $$B_{N^*_{\gm}}=\rm{conv}(B_{N^*_{\gm_1}} \cup B_{N^*_{\gm_2}}).$$
\end{proof}
\vspace{0,3cm}

\begin{definition}
Let $\gm_1$ and $\gm_2$ two collections of curves on $\sg_g$. We say that ~$N_{\gm_1}\leq ~N_{\gm_2}$ holds if the inequality holds as functions of $H_1(\sg_g,\R)$. This is equivalent to have $B_{N^*_{\gm_1}}\subset B_{N^*_{\gm_2}}$.  
\end{definition}

\begin{lemma}
If $\gm_1$ is a collection obtained by smoothing $\gm$ at a point $p$, then $$N_{\gm_1}\leq N_{\gm}.$$
\end{lemma}
\begin{proof}
By Lemma \ref{smt}, $B_{N^*_{\gm_1}}\subset B_{N^*_{\gm_1}}$.
\end{proof}\vspace{0,2cm}

\begin{definition}
A filling collection $\gm$ is \textbf{\textit{one-faced}} (respectively \textit{\textbf{two-faced}}) if $\sg_g-\gm$ is a disk (two disks).
\end{definition}

For a graph defined by a collection of curves with only double points, we have $|E|=2|V|$. Then, the Euler characteristic of the surface is:
$$\chi(\sg_g)=2-2g=|F|-|V|.$$
It follows that for a filling collection, we have $|V|=|F|+2g-2\geq 2g-1$. Therefore, the minimum is obtained for one-faced collections. In particular, in genus two, a one-faced collection has self-intersection number equal to $3$.
\begin{definition}
A norm $N_{\gm}$ is \textit{\textbf{even}} if it is so as a function of $H_1(\sg_g,\Z)$; otherwise, it is \textit{\textbf{odd}}.
\end{definition}\vspace{0,5cm}

The following lemma is one of the cornerstone of this article.
\begin{lemma}[\textit{Lower bound for intersection norms}]\label{lowerbound}
Every intersection norm defined by a filling collection is bounded from below by a norm defined by a two-faced collection. Moreover, if the norm is odd, it is bounded from below by a one-faced collection.
\end{lemma}
\begin{proof}
Let $\gm$ be a filling collection on $\sg_g$. We assume that $|F(\gm)|\geq 2$. Let  $p$ a double point of $\gm$ such that two different faces, say $F_1$ and $F_2$, are opposed at $p$.

By smoothing $\gm$ at $p$ so that the faces $F_1$ and $F_2$ are joined (see Figure~\ref{addit}), we define a new filling collection with one  face less. Step by step, following this process of smoothing intersection points at which two different faces are opposed, we reach a filling collection $\gm_n$ on which opposed faces at double point are the same and satisfy $N_{\gm_n}\leq N_{\gm}$.

Now, let $p$ be one of the double points of $\gm_n$, $e:=(v_1=p, v_2,...., v_n=p)$ a Eulerian cycle based at $p$, $F_a$ and $F_b$ the two faces at $p$ such that when we turn around $p$, we read $F_a-F_b-F_a-F_b$. 

The faces $F_a$ and $F_b$ are again the faces at the vertex $v_2$ as $p:=v_1$ and $v_2$ share a common edge. So, the point $v_2$ has the same configuration $F_a-F_b-F_a-F_b$. 

Since $v_3$ shares an edge with $v_2$ the faces $F_a$ and $F_b$ are again the faces at ~$v_3$, and $v_3$ again has the same configuration of faces. By applying this process step by step at each vertex of the Eulerian cycle, we show that around every vertex, we have the configuration $F_a-F_b-F_a-F_b$. Then $\gm_n$ has one or two faces according to whether $F_a$ is equal to $F_b$ or not. 

Suppose that $\gm_n$ is two-faced without any possibility of reduction. Then by the above argument any edge of $\gm_n$ separates two different faces. Therefore, $N_{\gm_n}$ is even.  In fact, if $\gamma$ is a transverse curve to $\gm_n$, then $\gamma$ alternates between $F_a$ and $F_b$ at each intersection with $\gm$. It implies that $i_{\gm_n}(\gamma)$ is even. So is ~$i_{\gm}(\gamma)$ since smoothing does not change the parity of the geometric intersection. 

Finally, an odd norm reduces to a one-faced collection.  
\begin{figure}[h!]
\begin{center}
\includegraphics[scale=0.15]{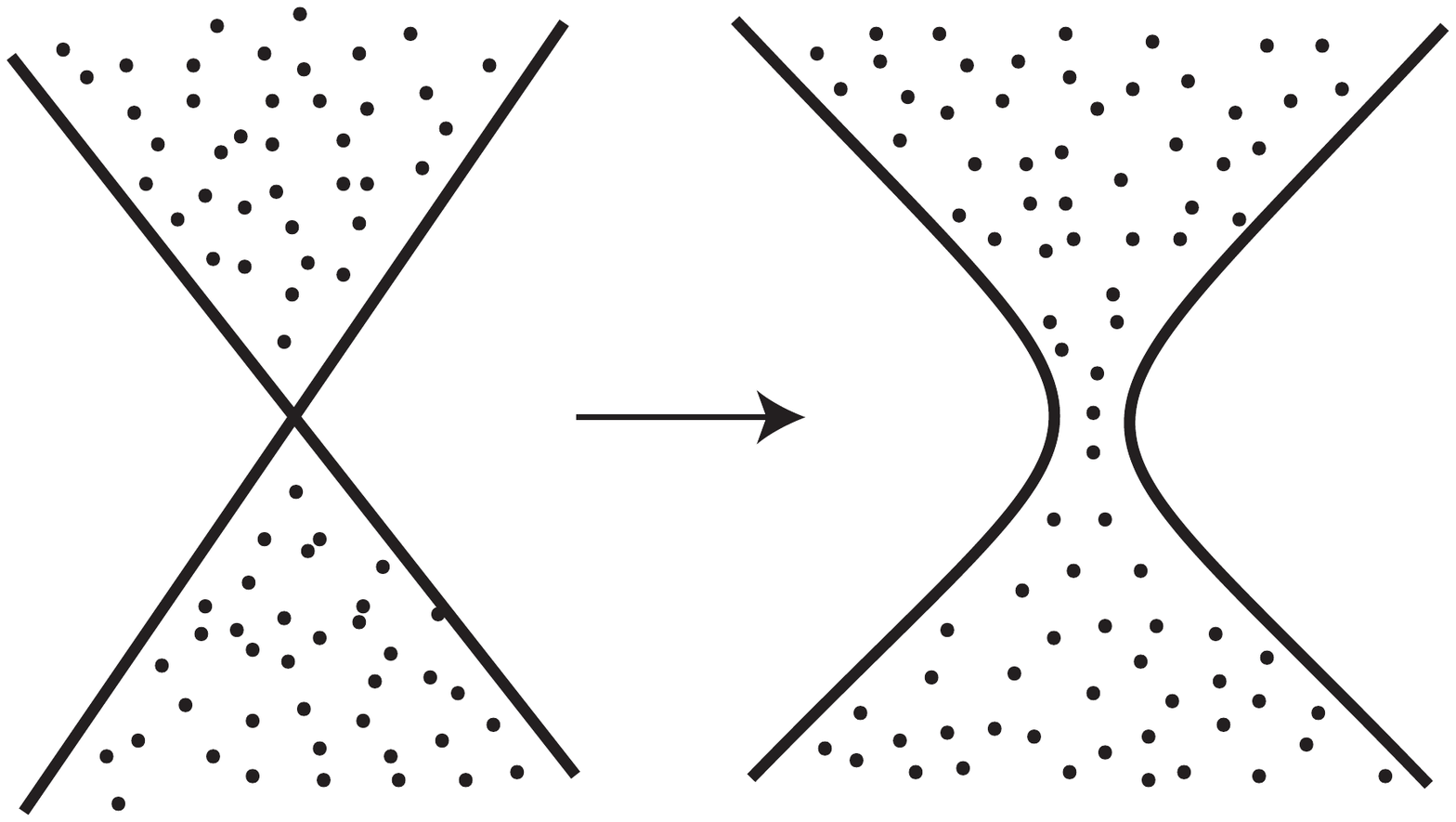}
\caption{Addition of two different faces.}
\label{addit}
\end{center}
\end{figure}
\end{proof}


\begin{corollary}\label{incube}
Every intersection norm with dual unit ball in the cube ~$[-1,1]^{2g}$ is bounded from below by a norm defined by a one-faced collection.
\end{corollary}
\begin{proof}
If  $B_{N^*_{\gm}}$ is a sub-polytope of $[-1,1]^{2g}$, then $N_{\gm}$ is odd and applying Lemma~\ref{lowerbound}, we obtain the result. 
\end{proof}
\end{subsection}
\end{section}

\begin{section}{Orbits of one-faced collections with dual unit ball in the cube $[-1,1]^4$}\label{sect3} 

Now, we  show that there are four orbits, under the mapping class group action, of one-faced collections with dual unit ball in the cube $[-1,1]^4$ (Theorem~\ref{counting}). 

\begin{paragraph}{Partial configuration.} We consider a collection of closed (non oriented) curves $\gm=\{\gamma_1,.....\gamma_n\}$ in $\sg_2$ whose complement is one disk. 

In what follows $\alpha_1, \beta_1, \alpha_2$ and $\beta_2$ are the oriented simple closed curves, that canonically represent the generators of the first homology group (see Figure~\ref{basesym}). Let $\eta:=\alpha_1\beta_1\alpha_1^{-1}\beta_1^{-1}$ be the curve depicted in red and $A_{\eta}$ be a tubular neighborhood of $\eta$.   
\begin{figure}[h!]
\begin{center}
\includegraphics[scale=0.18]{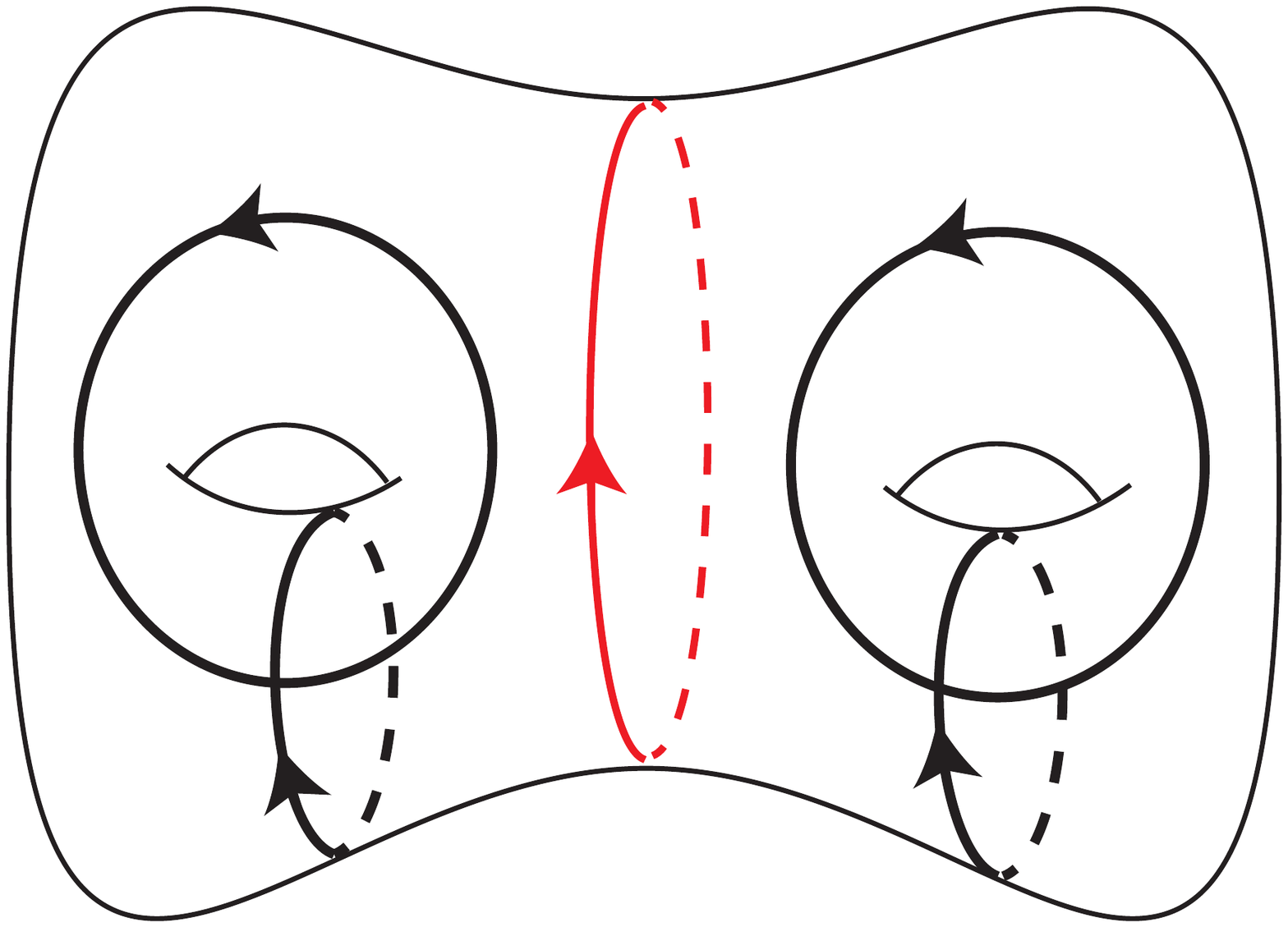}
\put(-28,13){$\alpha_1$}
\put(-13,13){$\alpha_2$}
\put(-30,34){$\beta_1$}
\put(-13,34){$\beta_2$}
\put(-20,37){$\eta$}
\caption{Canonical symplectic basis.}
\label{basesym}
\end{center}
\end{figure}

The following lemma gives a canonical partial configuration for one-faced collections. 

\begin{lemma}\label{config} If $\gm$ is a one-faced collection on $\sg_2$ with dual unit ball in the cube $[-1,1]^4$, then there exists a diffeomorphism $\psi$ of $\sg_2$ such that $$i(\alpha_i,\psi(\gm))=i(\beta_i,\psi(\gm))=1; i=1,2.$$

Hence, up to diffeomorphism and outside $A_{\eta}$, $\gm$ looks like in Figure~\ref{pconfig}.  
\end{lemma}
\begin{figure}[h!!]
\begin{center}
\includegraphics[scale=0.17]{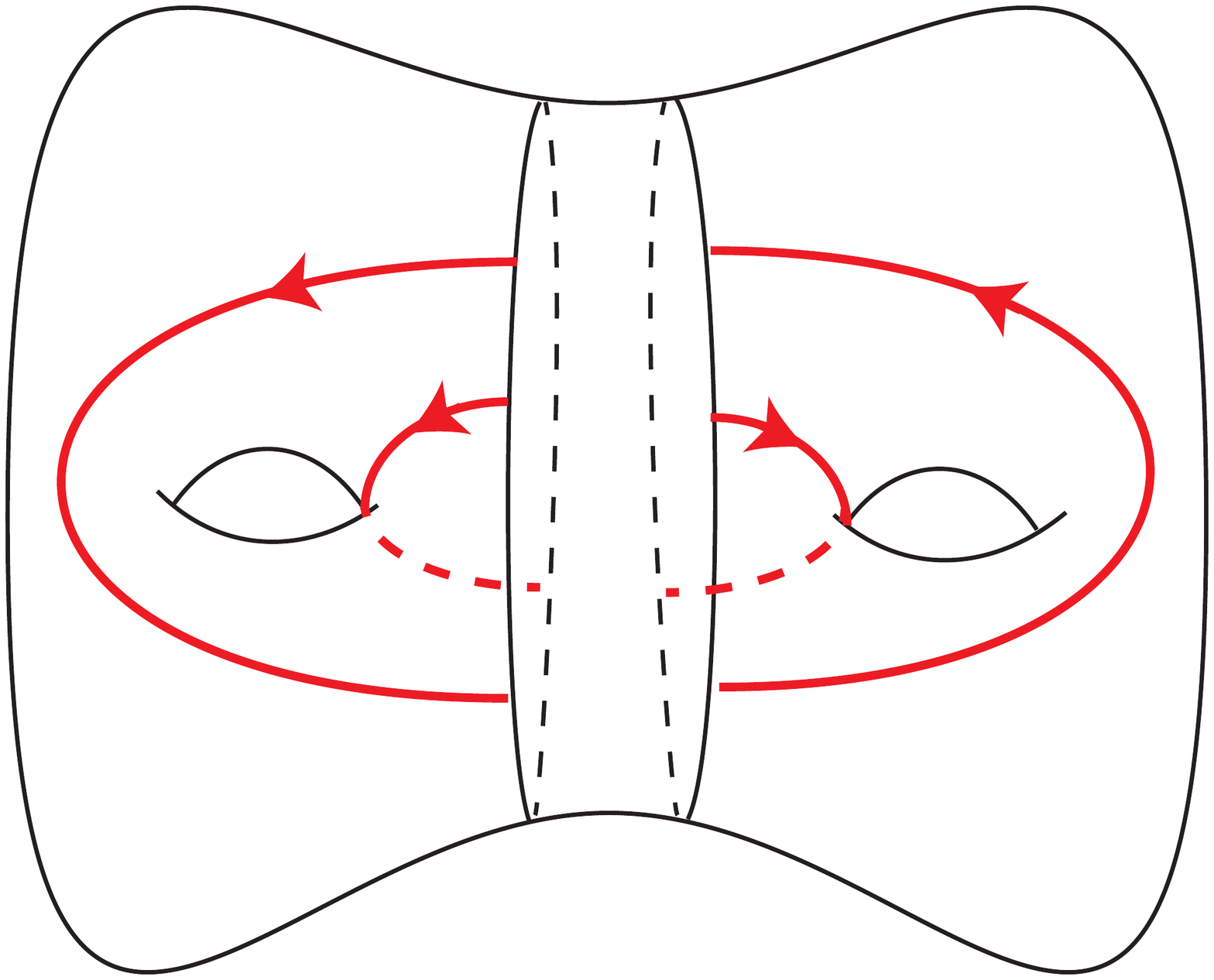}
\put(-24,27){$a_1$}
\put(-34,28){$b_1$}
\put(-14,27){$a_2$}
\put(-5,28){$b_2$}
\caption{Partial configuration of the collection $\psi(\gm)$; with labelled arcs.}
\label{pconfig}
\end{center}
\end{figure}

\begin{proof}
Since $\gm$ is one-faced with dual unit ball in $[-1,1]^4$ then $$N_{\gm}(a_i)=N_{\gm}(b_i)=1$$
where $a_i$, $b_i$ is the symplectic basis of $H_1(\sg_g,\R)$.
Now, as $N_{\gm}(a_i)=N_{\gm}(b_i)=1$, there is an oriented simple closed  curve $\alpha$ such that $$i(\alpha,\gm)=1$$ and $$[\alpha]=a_1.$$

Up to diffeomorphism, we can take $\alpha=\alpha_1$. 

Now, let $\beta$ be the $\gm$-minimizing simple curve in the homology class of $b_1$, then
$$i_a(\alpha_1,\beta)=1$$ and $$i(\gm,\beta)=1.$$
Therefore, one can make a surgery on $\beta$ along $\alpha$ (See Figure \ref{smoothing}) to get a new curve $\beta'$ such that $\beta'$ is a simple $\gm$-minimizing curve in the same homology class with $\beta$ and such that $$i(\beta',\alpha)=i(\beta',\gm)=1.$$
Up to diffeomorphism, we can take $\beta'=\beta_1$.

If $\alpha$ and $\beta$ are $\gm$-minimizing in the homology classes of $\alpha_2$ and $\beta_2$ respectively, we have 

$$i_a(\alpha,\alpha_1\cup\beta_1)=i_a(\beta,\alpha_1\cup\beta_1)=0.$$

Again, by performing surgery on $\alpha$ and $\beta$, we get $\alpha'$ and $\beta'$ such that
$$i(\alpha',\alpha_1\cup\beta_1)=i(\beta',\alpha_1\cup\beta_1)=0$$ and $$i(\alpha',\beta')=1.$$

Then, up to diffeomorphism $\alpha'=\alpha_2$ and $\beta'=\beta_2$.\\
This prove that up to diffeomorphism, $(\alpha_1,\beta_1,\alpha_2,\beta_2)$ are $\gm$-minimizing.
\end{proof}

\begin{figure}[htbp]
\begin{center}
\includegraphics[scale=0.15]{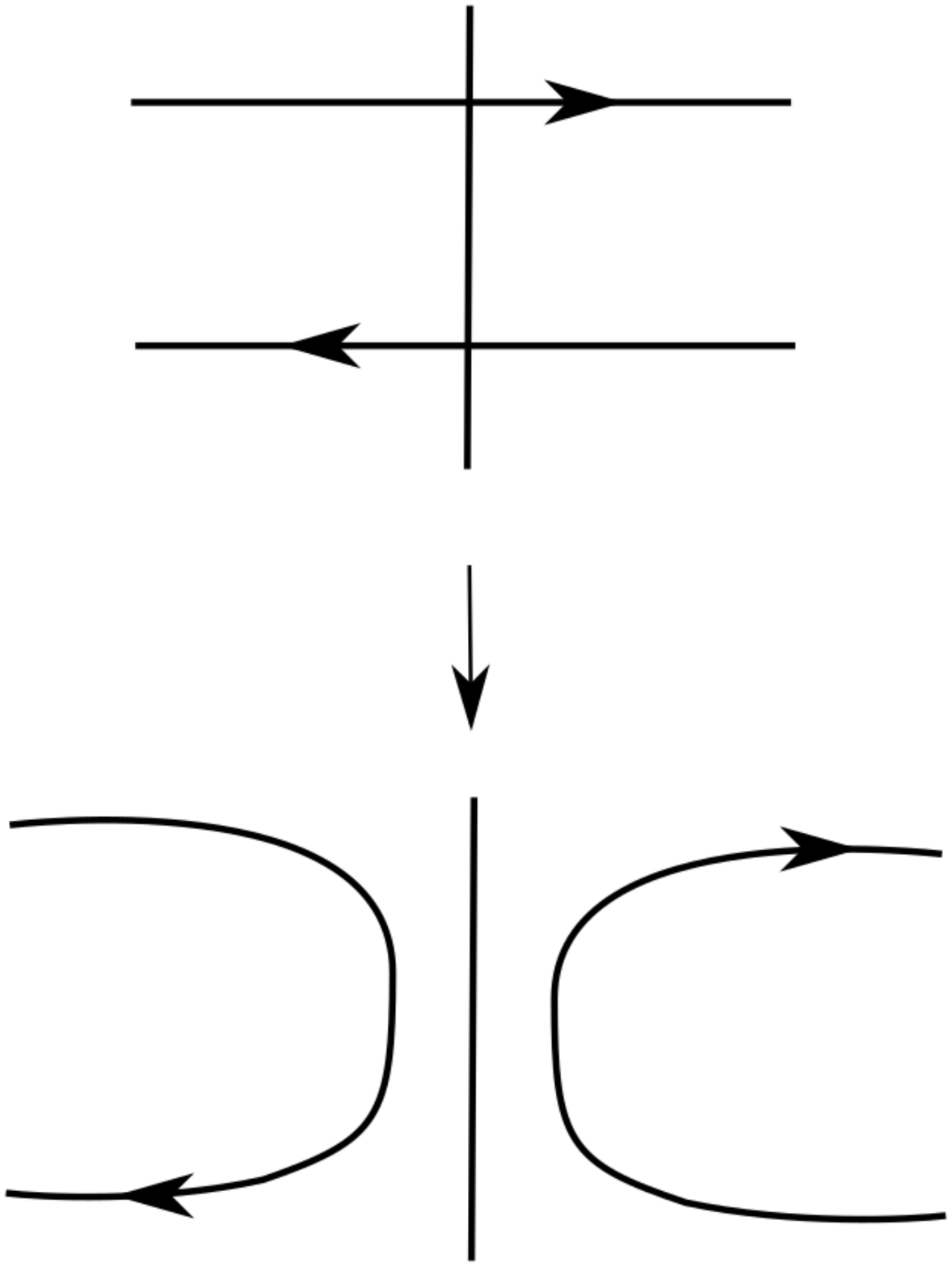}
\caption{Surgery along the vertical curve.}
\label{smoothing}
\end{center}
\end{figure}

\begin{remark} Lemma~\ref{config} remains true on a genus $g$ surface and the proof is the same.
\end{remark}
Lemma~\ref{config}  implies that, up to diffeomorphism,  a one-faced collection with dual unit ball in the cube $[-1,1]^4$ is obtained by connecting the extremities of the partial configuration by arcs in the annulus $A_{\eta}$. Moreover, the self-intersection number of $\gm$ is determined by the intersection between those arcs we used to complete the partial configuration. 

Let $a_1, b_1, a_2$ and $b_2$ be the four oriented arcs in the partial configuration (see Figure ~\ref{pconfig}). A closed curve from the partial configuration will be labelled by the arcs being used and the number of twists we make around $\eta$ when we walk along that curve. For instance, $a_1\eta^{2}b_1^{-1}b_2$ is the closed curve depicted on Figure~\ref{exmpcurv}. 
\begin{figure}[htbp]
\begin{center}
\includegraphics[scale=0.2]{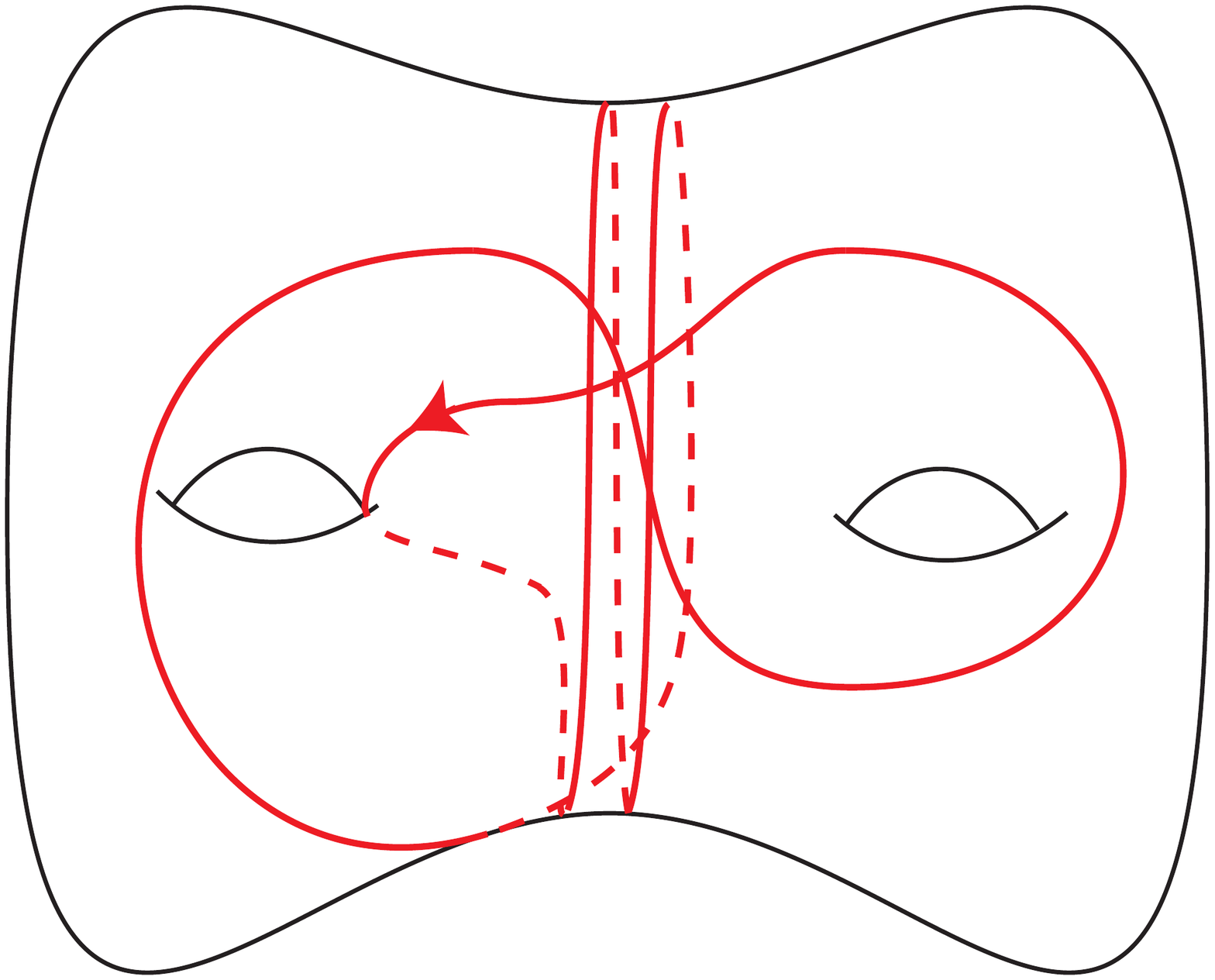}
\caption{The curve $a_1\eta^2b_1^{-1}b_2$}
\label{exmpcurv}
\end{center}
\end{figure}

As we are dealing with non oriented curves, the labeling of curves is defined up to cyclic permutation and reversing. For example, $a_1\eta^2b_1^{-1}b_2$ and $a_1^{-1}b_2^{-1}b_1\eta^{-2}$ are labels of the same curve. 
\end{paragraph} 
 
\begin{paragraph}{Intersection of arcs in an annulus:} As we said above, the geometric intersection of a one-faced collection is completely determined by the intersection of arcs in an annulus. Here, the intersection number is computed over the homotopy class of arcs with fixed end points. Now, let $\lambda$ be a simple oriented arc joining the two boundaries of $A$. Cutting along $\lambda$, we obtain a rectangle with two opposite sides identified. Let $X$ and $Y$ be two points in the boundary components of $A$. An oriented arc from $X$ to $Y$ will be denoted by $\overset{\rightarrow}{XY_p}$ where $p\in\Z$ is the algebraic intersection between $\overset{\rightarrow}{XY_p}$ and $\lambda$. 

\begin{figure}[htbp]
\begin{center}
\includegraphics[scale=0.17]{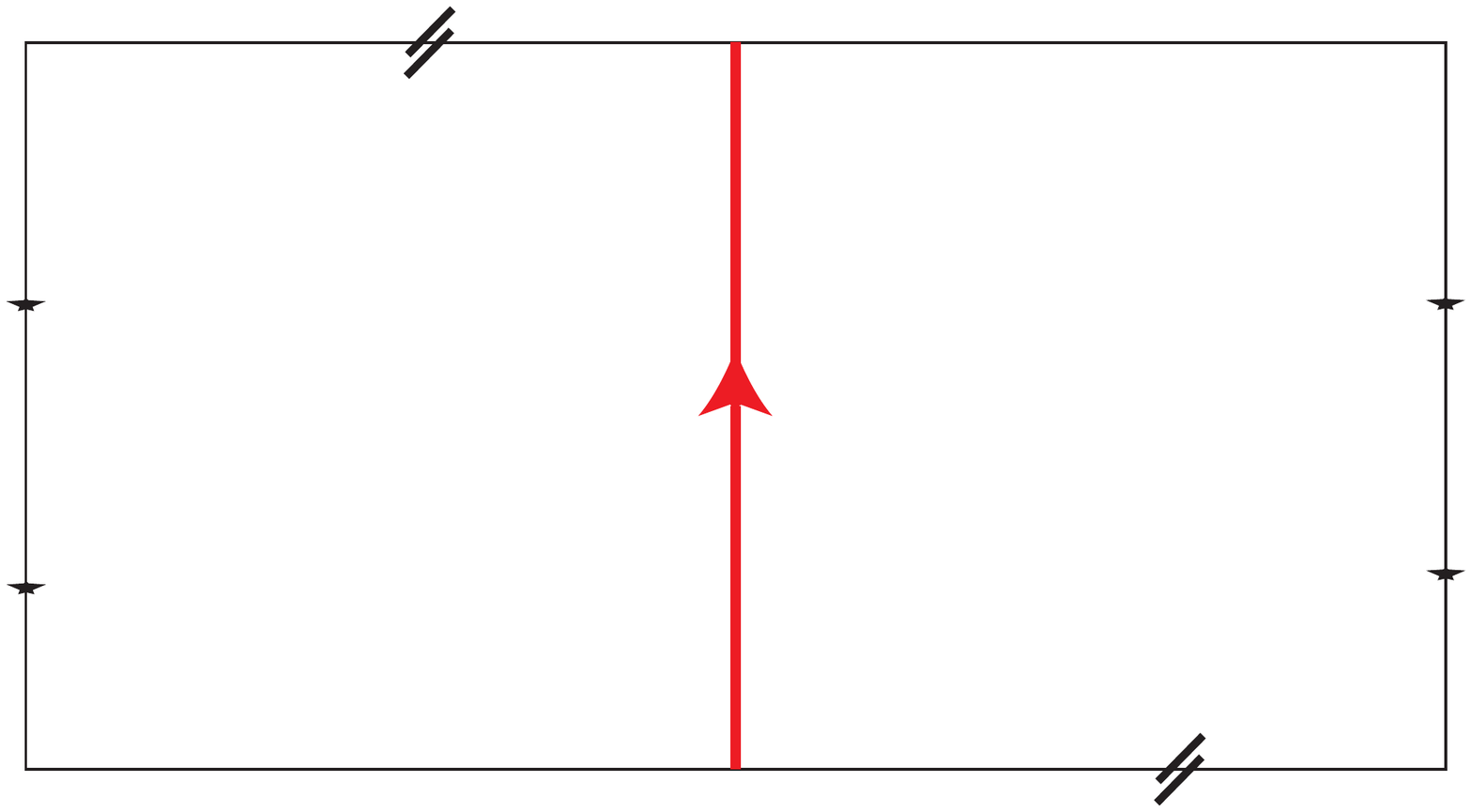}
\put(-37,26){\tiny{A}}
\put(-37,20){\tiny{C}}
\put(0,26){\tiny{B}}
\put(0,20){\tiny{D}}
\put(-16,23){$\eta$}

\caption{End-points in annulus.}
\label{anl}
\end{center}
\end{figure}

Let $A, B, C$ and $D$ four points in the boundaries of $A$ as in Figure~\ref{anl}.

\begin {lemma}\label{intersection}
The following formulas give the intersection between two oriented arcs in $A$:
\begin{itemize}
\item $i(\overset{\rightarrow}{AB}_p,\overset{\rightarrow}{CD}_q)=i(\overset{\rightarrow}{BA}_p,\overset{\rightarrow}{DC}_q)=|p-q|$
\item $i(\overset{\rightarrow}{AB}_p,\overset{\rightarrow}{DC}_q)=i(\overset{\rightarrow}{BA}_p,\overset{\rightarrow}{CD}_q)=|p+q|$
\item $i(\overset{\rightarrow}{AD}_p,\overset{\rightarrow}{CB}_q)=i(\overset{\rightarrow}{DA}_p,\overset{\rightarrow}{BC}_q)=|p-q-1|$
\item $i(\overset{\rightarrow}{AD}_p,\overset{\rightarrow}{BC}_q)=|p+q-1|$
\item $i(\overset{\rightarrow}{DA}_p,\overset{\rightarrow}{CB}_q)=|q+p+1|$
\end{itemize}
\end{lemma} 

\begin{proof}
Up to the Dehn twist $\tau_{\eta}^{-q}$ on the configuration of the arcs, one can assume that $q$ is equal to $0$ in all cases, that is one the arc is untwisted.

Therefore, we have:

$$i(\overset{\rightarrow}{AB}_p,\overset{\rightarrow}{CD}_q)=i(\overset{\rightarrow}{AB_{p'}},\overset{\rightarrow}{CD})= |p'|$$
with $$\overset{\rightarrow}{AB}_{p'}=\tau_{\eta}^{-q}(\overset{\rightarrow}{AB_p}).$$ Moreover, $p'=p-q$. Hence, we obtain the result.\vspace{0,5cm}

Again, for the second formula, we have:
$$i(\overset{\rightarrow}{AB}_p,\overset{\rightarrow}{DC}_q)=i(\overset{\rightarrow}{AB_{p'}},\overset{\rightarrow}{CD})=|p'|$$ and $p'=p+q$.
The difference between the first two cases show how crucial is the orientation for the computing of intersection. 

We treat the third case, the other are done in a similar way.\\ 
We still have that  $$i(\overset{\rightarrow}{AD}_p,\overset{\rightarrow}{CB}_q)=i(\overset{\rightarrow}{AD_{p'}},\overset{\rightarrow}{CB})= |p'-1|$$ and $$p'=p-q.$$
The appearance of $-1$ in this case comes from the cross configuration of the extremities.
\end{proof} 
\end{paragraph} 

\begin{paragraph}{List of one-faced collections with dual unit ball in the cube $[-1,1]^4$\string:} Now, we are able to count all one-faced collections whose dual unit ball is a sub polytope of the cube $[-1,1]^4$. Before that, we define some diffeomorphisms which will be useful for the proof.

If $\gamma$ is an oriented simple closed curve on $\sg_2$, we recall that $\tau_{\gamma}$ is the right-handed Dehn twist along $\gamma$.

\begin{figure}[h!]
\begin{center}
\includegraphics[scale=0.2]{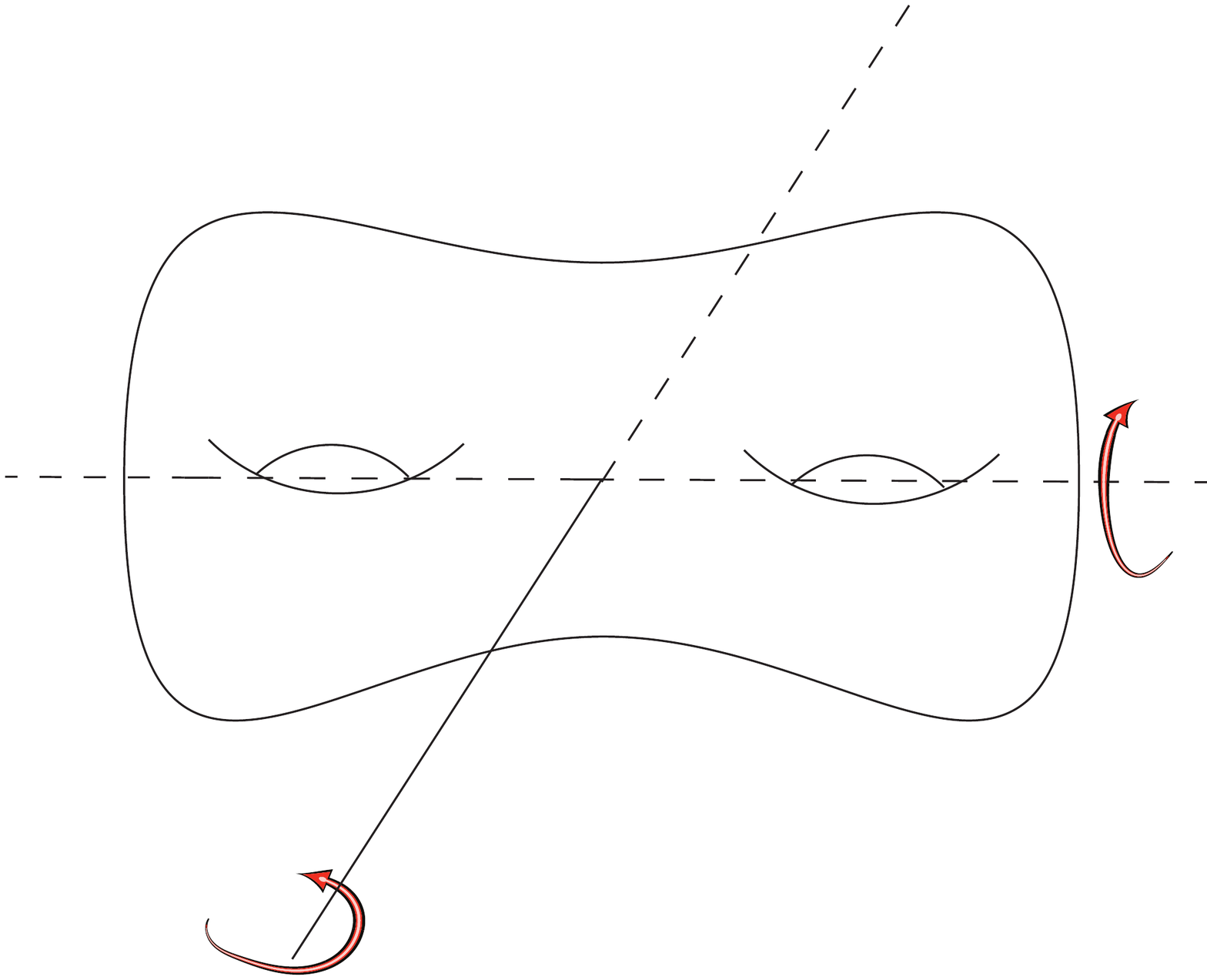}
\put(-14,45){$\mathcal{D}$}
\caption{Rotations $R_1$ and $R_2$}
\label{rot}
\end{center}
\end{figure}

Let $R_1$ (respectively $R_2$) be the rotation of angle $\pi$ along the axis $\mathcal{D}$ (respectively the horizontal axis) as depicted in Figure~\ref{rot}. The diffeomorphism ~$R_1$  (respectively $R_2$) is an involution and it maps $\alpha_1$ to $\alpha_2$, $\beta_1$ to $\beta_2$ and $\eta$ to ~$\eta^{-1}$ (respectively $\alpha_i$ to $\alpha_i^{-1}$, $\beta_i$ to $\beta_i^{-1}$ and $\eta$ to $\eta$).

We recall that $\alpha_i$ and $\beta_i$ can be interchanged by a diffeomorphism. More precisely, there is a diffeomorphism sending $\alpha_i$ to $\beta_i$ and $\beta_i$ to $\alpha_i^{-1}$. This fact implies that in the writing of the label of the curves, $a_i$ can be replaced by $b_i$ and $b_i$ by $a_i^{-1}$; we call this operation \textit{\textbf{interchanging}}.

\begin{definition}
Let  $\gm$ be a collection of closed curves on $\sg_2$. A cycle $\gamma$ in $\gm$ ($\gm$ seen as graph on $\sg_2$) is separating if $\sg_2-\gamma$ has more than one component.
\end{definition}

The following lemma gives a necessary condition for a collection to be one-faced.
\begin{lemma}\label{cycle}
If $\gm$ is a one-faced collection, then $\gm$ does not contain a separating cycle.
\end{lemma}   
\begin{proof}
Assume that $\gm$ contain a separating cycle $\gamma$, then $\sg_2-\gamma$ has at least two connected components. We have in this case more than one disc in the complement.
So if $\gm$ is one-faced, it does not contain a separating cycle.  
\end{proof}

Now, we can state the main result of this section which is an elaborate form of Theorem \ref{counting}.
  
  \begin{thm}[Orbits of one-faced collections]\label{countfin}
If $\gm$ is a one-faced collection on ~$\sg_2$ with dual unit ball in the cube $[-1,1]^4$, then $\gm$ has at most three closed curve. Moreover, up to diffeomorphism, 
\begin{itemize}
\item if $\gm$ is made of three closed curves, then $\gm=\{a_1,a_2,b_1b_2^{-1}\}$
\item if $\gm$ is made of two closed curves, then $$\gm=\{a_1a_2^{-1}, b_1b_2\eta\}\hspace{0,2cm} \rm{or}\hspace{0,2cm} \gm=~\{a_1, b_1b_2\eta a_2\}$$
\item if $\gm$ is made by one closed curve, then $\gm=\{a_1a_2^{-1}b_1^{-1}b_2\eta\}$
\end{itemize} 
  \end{thm}
  \begin{proof}
  If $\gm$ is one-faced,  then $i(\gm,\gm)=3$ (its comes from an Euler characteristic argument; cf Section \ref{sect2}). 

Now, if $\gm$ has at least four closed curves, then the arcs $a_i, b_i (i=1, 2)$ belong to four different closed curves  $\alpha_i\eta^{p_i}, \beta_i\eta^{q_i}$; otherwise $\gm$ would contain a separating cycle. Therefore,  $i(\gm,\eta)=0$ which is absurd as $\gm$ is filling. So, if $\gm$ is one-faced $|\gm|\leq 3$.\vspace{0,5cm} 

\textbf{Case $1$:} If $|\gm|=3$, then two arcs of the partial configuration belong to the same closed curve and the others two belong to two different closed curves. Moreover, as $\gm$ is filling, the two arcs containing in the same closed curve are in different handles. As one can interchange $a_i$ and $b_i$, we can assume that the curve containing two arcs is $\gamma:=b_1\eta^pb_2^{-1}\eta^q$; the other curves being $\alpha_1\eta^r$ and $\alpha_2\eta^s$. Since $\gm$ is one-faced, it does not contain a separating cycle that is $r=s=0$, and up to a Dehn twist along $\eta$, one can take p=0 that is  $\gamma=b_1b_2^{-1}\eta^q$. The fact that $i(\gm,\gm)=3$ implies that $$i(\gamma,\gamma)=1.$$ By Lemma~\ref{intersection} $i(\gamma,\gamma)=|q+1|=1$; it implies that $q=0$ or $q=-2$ and one check that $\gm_1=\{a_1,a_2,a_1a_2^{-1}\}$ and  $\gm_2=\{a_1,a_2,b_1b_2^{-1}\eta^{-2}\}$ are in the same orbit under the mapping group action.
\begin{figure}[htbp]
\begin{center}
\includegraphics[scale=0.2]{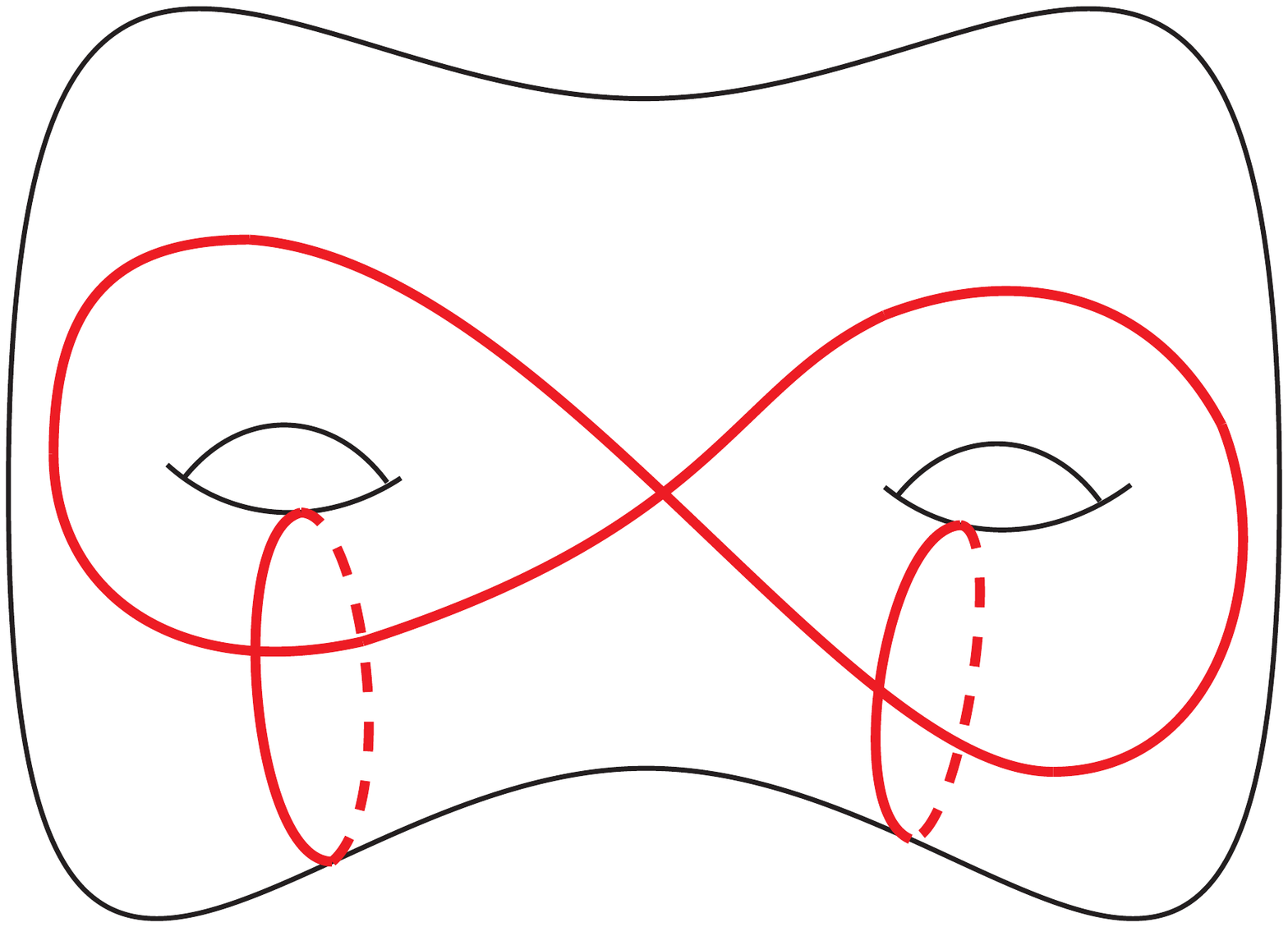}
\caption{One-faced collection with three curves}
\label{threecurv}
\end{center}
\end{figure}

\textbf{Case 2:} If $|\gm|=2$, then one of the curves of $\gm$ is simple. Otherwise if the two curves are not simple, one can smooth intersection point of one of the curves in $\gm$ --let $g_1$ be that curve-- such that each smoothing separate $g_1$ in to two component. we obtain at least two simple curves $\lambda_j$, $j=1,..,n$. The curves $\lambda_j$, as they are all parallels to $g_1$,  intersects $g_2$. Then $$i(\gm,\gm)\geq i(g_1,g_1)+i(g_2,g_2)+ \displaystyle{\sum_j{i(\lambda_j,g_2)}}>3$$
which is absurd since $\gm$ is one-faced. Therefore, one of the two curves is simple, say $g_1$. Moreover, we know from \cite{Aoug} that $\sg_2$ does not admit a filling pairs (i.e a one faced-collection making of two simple closed curves), so the other one is non simple.

Up to diffeomorphism (interchanging and rotations), one can assume that $a_1$ is contained in $g_1$. 

\textbf{Case 2.1:} If $g_1$ does not contain another arc, then $$g_2=b_1b_2\eta^pa^{\varepsilon}$$ with $\varepsilon=\pm1$. In this case, 
$$i(\gm,\gm)=i(g_1,g_2)+i(g_2,g_2)$$
 and  
 $$i(g_1,g_2)=1.$$ Its implies that $i(g_2,g_2)=2$. The solution of this equation is $p=1$ \\and $\varepsilon=1$.

So $$\gm=\{a_1,b_1b_2\eta a_2\}$$ which indeed is a one-faced collection(see Figure \ref{twocurv}).\vspace{0,5cm} 

\textbf{Case 2.2:} If $g_1$ contains another arc than $a_1$, that arc cannot be in the same handle as $a_1$ (otherwise, the filling condition would fail). Up to interchanging, one can suppose that $$g_1=a_1\eta^pa_2^{-1}\eta^{-p}$$ and again by applying a Dehn twist around ~$\eta$, one can take $g_1=a_1a_2^{-1}$ and  $g_2=b_1\eta^pb_2^{\varepsilon}\eta^q$ with $\varepsilon=\pm1$. Moreover, $$i(\gm,\gm)=i(g_1,g_2)+ i(g_2,g_2).$$

We have $i(\alpha_1\cup\alpha_2,\beta_1\cup\beta_2)\equiv i(g_1,g_2) \rm{mod}\hspace{0,1cm} 2$ since $\alpha_1\cup\alpha_2$ (respectively $\beta_1\cup\beta_2$) is homologous to $g_1$ (respectively $g_2$). It implies that $$i(g_1,g_2)=2$$ and, $$i(g_1,g_2)=1.$$

\textbf{Case 2.2.1:} If $\varepsilon=-1$, by applying the formulas of Lemma~\ref{intersection}, we have: $$i(g_2,g_2)=|p+q+1|$$ and $$i(g_1,g_2)=|p|+|q|+|q+1|+|p+1|.$$

The solution of the equations $i(g_2,g_2)=1$ and $i(g_1,g_2)=2$ are $\{p=0,q=0\}$ and $\{p=-1,q=-1\}$. The two collections obtained are not filling since $i(b_1b_2,\gm)=0$.\vspace{0.5cm} 

\textbf{Case 2.2.2:} If $\varepsilon=1$, then $i(g_2,g_2)=|p-q|$ and $i(g_1,g_2)=2(|p|+|q|)$. The solution of the equations $i(g_2,g_2)=1$ and $i(g_1,g_2)=2$ are $\{p=0, q=\pm1\}$ and $\{p=\pm1, q=0\}$. 

We check that $\gm_1=\{a_1a_2^{-1}, b_1\eta^{\pm1}b_2\}$ and $\gm_2=\{a_1a_2^{-1}, b_1b_2\eta^{\pm1}\}$ are one-faced (here, $\gm_i$ is a union of two collection according on whether the power of $\eta$ is $1$ or $-1$ ). The rotation $R_1$ maps elements $\gm_1$ to elements of $\gm_2$. Finally, the collection $\{a_1a_2^{-1}, b_1b_2\eta\}$ is the mirror image of $\{a_1a_2,b_1b_2\eta^{-1}\}$.\vspace{0,5cm}

Hence, up to diffeomorphism, we have two one-faced collections with two curves (see Figure~\ref{twocurv}) namely $$\gm_1=\{a_1a_2^{-1}, b_1b_2\eta\}$$ and $$\gm_2=\{a_1,b_1b_2\eta a_2\}$$

\begin{figure}[htbp]
\begin{center}
\includegraphics[scale=0.2]{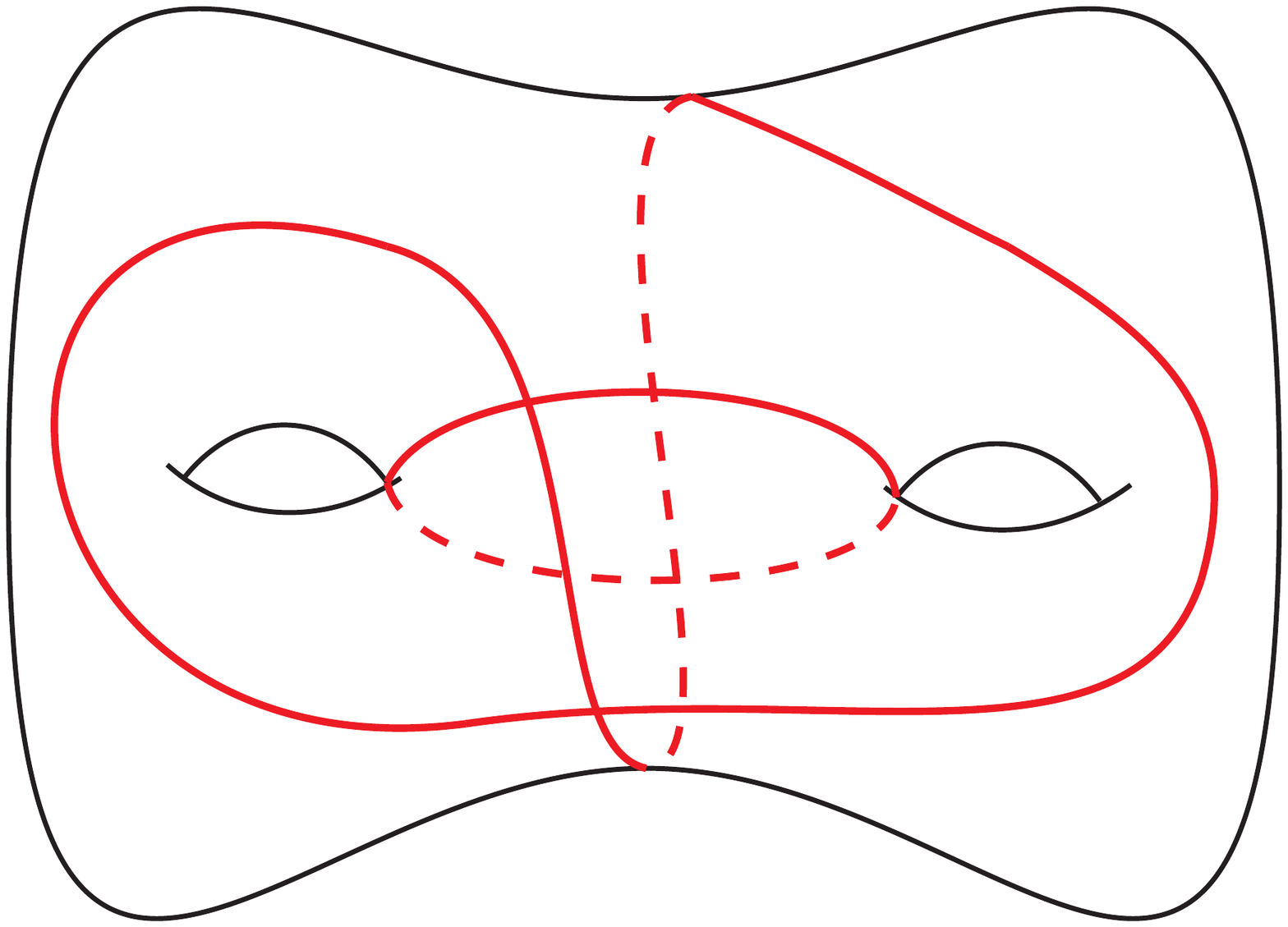}
\includegraphics[scale=0.2]{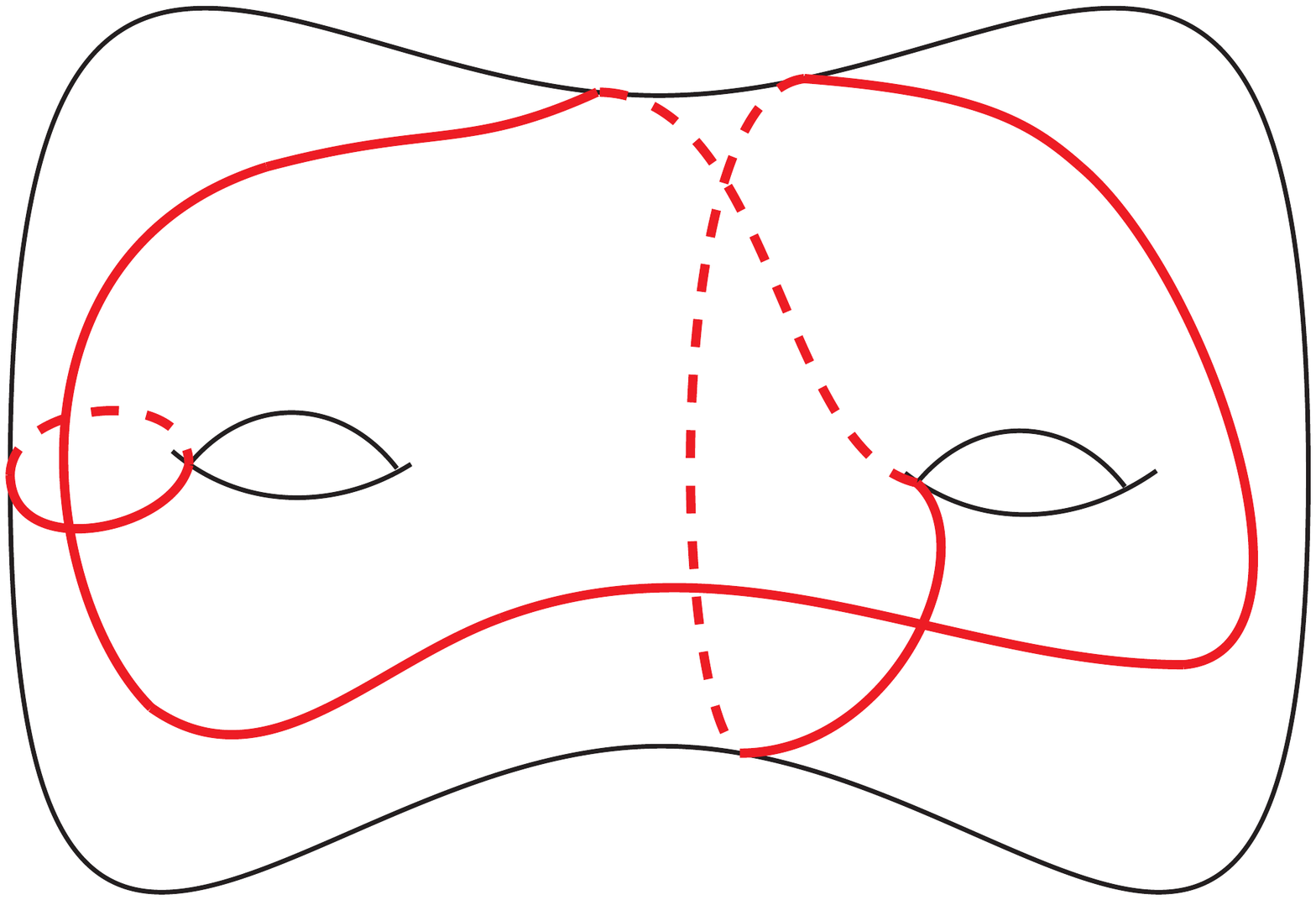}
\caption{One-faced collections with two curves}
\label{twocurv}
\end{center}
\end{figure}

\textbf{Case 3:} If $\gm$ has only one curve $g$, then up to diffeomorphism (interchangeability and rotations) $$g=a_1a_2^{-1}\eta^pb_1^{\varepsilon_1}\eta^qb_2^{\varepsilon_2}\eta^r$$ or $$g=a_1\eta^pb_1^{\varepsilon_1}a_2^{-1}\eta^qb_2^{\varepsilon_2}\eta^r,$$ where $\varepsilon_i=\pm1$. \vspace{1cm}

If $g=a_1\eta^pb_1^{\varepsilon_1}a_2^{-1}\eta^qb_2^{\varepsilon_2}\eta^r$, we check that $\gm$ is either not filling, either filling with more than one disk in its complement.\vspace{0,5cm} 

For $g(\varepsilon_1,\varepsilon_2)=a_1a_2^{-1}\eta^pb_1^{\varepsilon_1}\eta^qb_2^{\varepsilon_2}\eta^r$, $R_1$ sends $g(\varepsilon_1,\varepsilon_2)$ to $g(-\varepsilon_1,-\varepsilon_2)$.\\ If we start with $g=a_1a_2^{-1}\eta^pb_1^{-1}\eta^qb_2\eta^r$ and we change $a_1$ to $b_1$ by a diffeomorphism (that diffeomorphism will map $b_1$ to $a_1^{-1}$), $g$ gets mapped to $$g'=b_1a_2^{-1}\eta^pa_1\eta^qb_2\eta^r.$$ Now, if we reverse the orientation of $g'$ starting at $a_1$, we have $$g'=a_1^{-1}\eta^pa_2b_1^{-1}\eta^rb_2^{-1}\eta^q,$$ and $$R_2(g')=a_1\eta^pa_2^{-1}b_1\eta^rb_2\eta^q.$$ 

Finally, $\tau_{\eta^{-p}}\circ R_2(g')=a_1a_2^{-1}\eta^pb_1\eta^qb_2\eta^r.$\\
Hence, up to diffeomorphism, one can look at the case where $$\varepsilon_1=1;  \varepsilon_2=-1$$

In this case we have $$i(\gm,\gm)=|p|+|q|+|r|+|p+q+1|+|p-r|+|q+r+1|.$$

The equation $i(\gm,\gm)=3$ has two solutions $$\{p=0, q=0, r=-1\}$$ and $$\{p=-1, q=0, r=0\}.$$ The collections $\gm_1=\{a_1a_2^{-1}b_1b_2^{-1}\eta^{-1}\}$ and $\gm_2=\{a_1a_2^{-1}\eta^{-1}b_1b_2\}$ are one-faced. Moreover $R_1(\gm_1)=\gm_2$. Therefore, up to diffeomorphism, we have one one-faced collection with one curve (See Figure~\ref{onecurve}), namely $$\gm=\{a_1a_2^{-1}b_1b_2^{-1}\eta\}.$$ 

\begin{figure}[h!]
\begin{center}
\includegraphics[scale=0.2]{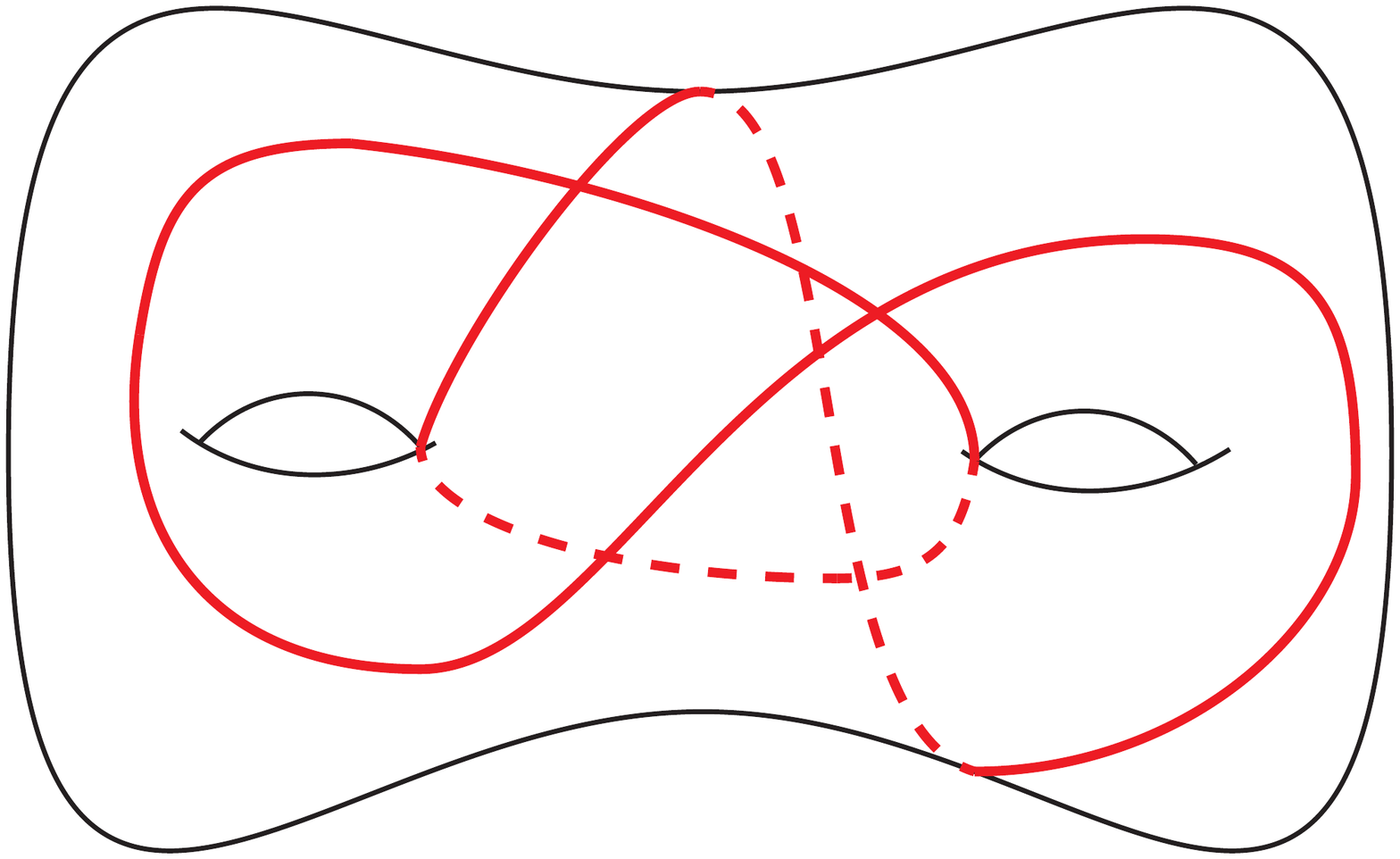}
\caption{One-faced collection made of one curve}
\label{onecurve}
\end{center}
\end{figure}
\end{proof}
\end{paragraph}
\begin{proof}[\textbf{Proof of Theorem \ref{thmprinc}}]
By Lemma~\ref{mincollect}, we can restrict our attention to minimal collections. By Corollary~\ref{incube}, if $P\in\mathcal{P}_8$ is the dual unit ball associated to a collection $\gm$, then $\gm$ is one-faced. Otherwise, if  $\gm$ had more than one face, it would have been possible to reduced $\gm$ to a filling collection $\gm'$ such that ~$N_{\gm'}\leq N_{\gm}$ that is $B_{N'^*_{\gm}}$ has less than eight vectors with non empty interior, which is impossible. 

It follows that $\gm$ is one of the collection in Theorem~\ref{countfin}. We check that dual unit balls of those collections are not in $\mathcal{P}_8$ (see bellow for their dual unit balls); which finally proves that elements of $\mathcal{P}_8$ are not realizable. 

\begin{paragraph}{Computation of dual unit balls:}
We compute the dual unit ball of an intersection norm by evaluating all Eulerian co-orientations on the canonical homology basis \cite{Direc}. Doing so, we obtain the vertices (of the dual unit balls) below: 
\begin{align*}
\{a_1,a_2,b_1b_2^{-1}\}&\mapsto [-1,1]^{4}\\
\{a_1a_2^{-1}, b_1b_2\eta\}&\mapsto\{\pm(1,1,1,-1); \pm(1,-1,1,1);\pm(1,1,1,1);\pm(1,1,-1,-1);\\
                                             &\hspace{1cm}\pm(1,-1,-1,1)\}\\
\{a_1, b_1b_2\eta a_2\}&\mapsto \{\pm(1,1,1,-1); \pm(1,-1,1,-1);\pm(1,1,-1,-1);\pm(-1,1,1,1);\\
                                             &\hspace{1cm}\pm(-1,1,-1,1)\}\\
\{a_1a_2^{-1}b_1^{-1}b_2\eta\}&\mapsto\{\pm(1,1,-1,-1); \pm(1,-1,-1,1);\pm(1,-1,1,1);\pm(1,1,1,1);\\
                                             &\hspace{1cm}\pm(-1,1,1,1)\}.
\end{align*} 
\end{paragraph}
The first collection has the whole unit cube as dual unit ball; the others three have dual unit balls with ten vectors.    
\end{proof}
\end{section}

UnitÈ de MathÈmatiques Pures et AppliquÈes (UMPA), ENS-Lyon.\\
\textit{E-mail address}: abdoul-karim.sane@ens-lyon.fr
 \end{document}